\documentclass{amsart}
\usepackage{amsfonts,amsmath,amsthm,amssymb}
\usepackage{mathrsfs}
\usepackage[dvipsnames]{xcolor}
\definecolor{MediumGray}{gray}{0.6}
\usepackage{tikz}
\usepackage{ytableau}
\usepackage[colorlinks=true,linkcolor=blue,citecolor=blue]{hyperref}

\usepackage{multirow}
\usepackage{comment}

\usepackage[bbgreekl]{mathbbol}

\numberwithin{equation}{section}
\theoremstyle{plain}
\newtheorem{theorem}[equation]{Theorem}
\newtheorem{prop}[equation]{Proposition}
\newtheorem{lemma}[equation]{Lemma}

\newtheorem{example}[equation]{Example}
\newtheorem{remark}[equation]{Remark}

\newcommand{\bs}{\mathbf{s}}

\newcommand{\cO}{\mathcal{O}}

\newcommand{\pl}{{\!+\!}}
\newcommand{\mn}{{\!-\!}}
\newcommand{\bfG}{\mathbf{G}}
\newcommand{\bfT}{\mathbf{T}}
\newcommand{\bfB}{\mathbf{B}}

\newcommand{\blambda}{\boldsymbol{\lambda}}
\newcommand{\bmu}{\boldsymbol{\mu}}
\newcommand{\bnu}{\boldsymbol{\nu}}
\newcommand{\brho}{\boldsymbol{\rho}}
\newcommand{\bfs}{\mathbf{s}}
\newcommand{\bsigma}{\boldsymbol{\sigma}}

\def\bbF{{\mathbb{F}}}
\def\bbZ{{\mathbb{Z}}}
\def\bbK{{\mathbb{K}}}
\def\bbk{{\mathbb{k}}}
\def\bbO{{\mathbb{O}}}
\def\bbQ{{\mathbb{Q}}}
\def\bbL{{\mathbb{\Lambda}}}

\newcommand{\nmod}{\textsf{-}\mathsf{mod}}
\newcommand{\umod}{\textsf{-}\mathsf{umod}}

\begin{document}
\title{Decomposition numbers for the principal $\Phi_{2n}$-block of $\mathrm{Sp}_{4n}(q)$ and $\mathrm{SO}_{4n+1}(q)$}
\author{Olivier Dudas}
\address[O.D.]{Universit\'{e} de Paris and Sorbonne Universit\'{e}, CNRS, IMJ-PRG, F-75006 Paris, France}
\email{olivier.dudas@imj-prg.fr}
\author{Emily Norton}
\address[E.N.]{Laboratoire de Math\'{e}mathiques Blaise Pascal, Universit\'{e} Clermont Auvergne, 3 Place Vasarely, 63178 Aubi\`ere, France}
\email{Emily.NORTON@uca.fr}

\begin{abstract}
We compute the decomposition numbers of the unipotent characters lying in the principal $\ell$-block of a finite group of Lie type $B_{2n}(q)$ or $C_{2n}(q)$ when $q$ is an odd prime power and $\ell$ is an odd prime number such that the order of $q$ mod $\ell$ is $2n$. Along the way, we extend to these finite groups the results of \cite{DVV19} on the branching graph for Harish-Chandra induction and restriction.
\end{abstract}

\maketitle


\section*{Introduction}

The representation theory of a finite group of Lie type $G:=\mathbf{G}(\mathbb{F}_q)$ over a field of positive characteristic $\ell$ coprime to $q$ has a close relationship to the representation theory of the Hecke algebra of its Weyl group. The decomposition matrix of the Hecke algebra always embeds as a submatrix of the decomposition matrix of $G$.  When $G=\mathrm{GL}_n(q)$ is the finite general linear group, the square unitriangular submatrix of the decomposition matrix of the unipotent blocks is the same as the decomposition matrix of the $q$-Schur algebra, a quasihereditary cover of the Hecke algebra of the symmetric group $S_n$. This is related to the fact that in characteristic $0$, $\mathrm{GL}_n(q)$ has exactly one cuspidal irreducible unipotent representation as $n$ ranges over $\mathbb{N}$, namely, the trivial representation of $\mathrm{GL}_1(q)$. When $G$ is not of type $A$, less is understood about the decomposition matrix of the unipotent blocks of $G$. There are more cuspidal unipotent representations in characteristic $0$ which give rise to multiple Hecke and quasihereditary algebras, all of which play a role in the unipotent blocks of $G$. However, in the general case, the knowledge of the decomposition numbers for these algebras is not enough to determine those of $G$. In \cite{Du13}  the first author initiated the use of Deligne--Lusztig characters to find the missing numbers. This proved successful in determining decomposition matrices for finite groups of Lie type in small rank, see \cite{HiNo14}, \cite{DuMa15}, \cite{DuMa16}, \cite{DuMa20}.

\smallskip

In this paper, we are concerned with groups of type $B_m$ and $C_m$ such as $G=\mathrm{SO}_{2m+1}(q)$ and $G=\mathrm{Sp}_{2m}(q)$ for odd $q$. If $n$ is the order of $q^2$ in $\mathbb{F}_\ell^\times$, the complexity of the decomposition matrix grows with $m/n$. For that reason we will consider the case where $m = 2n$, which is somehow the simplest case outside of the cyclic defect case. Two situations arise: \begin{itemize}
\item \emph{(linear prime case)} $q$ has order $n$ in $\mathbb{F}_\ell^\times$, in which case $n$ is necessarily odd, and the representation theory of $G$ behaves as a type $A$ phenomenon and can be deduced from the representation theory of $q$-Schur algebras of symmetric groups \cite{GrHi};
\item \emph{(unitary prime case)} $q$ has order $2n$ in $\mathbb{F}_\ell^\times$. Explicit decomposition matrices in that case were obtained by Okuyama--Waki for $m=2$ \cite{OkWa} and by Malle and the first author for $m=4$ \cite{DuMa16} and $m=6$ \cite{DuMa20}. 
\end{itemize}
Our main result provides a generalisation of the unitary prime case to any even $m$. To our knowledge it is the first general result for defect $2$ blocks of finite groups of Lie type outside of type $A$ phenomena. 

\smallskip

Let $\Phi_d(q)$ be the $d$-th cyclotomic polynomial evaluated at $q$. 

\smallskip
\noindent \textbf{Main Theorem.}
 (Theorems \ref{thm:B2 submatrix}, \ref{thm:B6 submatrix}, \ref{thm:principal series submatrix}, \ref{prop:PIM0}, \ref{thm:PIM1}, and \ref{thm:PIM2})
 \emph{Let $G$ be a finite group of Lie type over $\mathbb{F}_q$ of type $B_{2n}$ or $C_{2n}$ for $q$ an odd prime power. Let $\ell$ be an odd prime number such that the order of $q$ in $\bbF_{\ell}^\times$ is $2n$. Then all but two decomposition numbers of the unipotent characters in the principal $\ell$-block of $G$ are known. If $\Phi_{2n}(q)_\ell>4n$ then both these numbers are $2$ and the decomposition numbers are completely known.}
 
 \smallskip
 
Under our assumptions on $\ell$ and $q$, studying the unipotent characters in the principal $\ell$-block is a reasonable restriction. First, any other unipotent $\ell$-block has defect $0$ or $1$, and its decomposition matrix is known by \cite{FoSr90}. Second, the decomposition numbers of the non-unipotent characters in the principal $\ell$-block of $G$ may be recovered from those of the unipotent characters and from partial knowledge of the character table of $G$ by \cite{GeHi91,Ge93}.

\smallskip

The methods used to obtain the decomposition matrices are two-fold:
\begin{itemize}
  \item[(1)] First, we use Harish-Chandra induction and restriction to produce projective indecomposable modules (PIMs);
  \item[(2)] Second, we compute the missing PIMs (corresponding to cuspidal simple modules) using some partial information on the decomposition of Deligne--Lusztig characters on PIMs. 
\end{itemize}
For both of these steps, we use a truncated version of the Harish-Chandra induction and restriction coming from the categorical $\widehat{\mathfrak{sl}}_{2n}$-action on unipotent representations defined in \cite{DVV17}. The recent unitriangularity result in \cite{BDT19} allows us to compute the branching graph for this truncated induction, which provides the missing information for step (2) to be successful. 

\subsection*{Acknowledgments} We thank Gunter Malle and Raph\"{a}el Rouquier for helpful conversations, and Gunter Malle for perspicacious comments on a draft of this paper. O. Dudas gratefully acknowledges  financial support by the ANR, Project No ANR-16-CE40-0010-01E and by the grant SFB-TRR 195. E. Norton was supported by the grant SFB-TRR 195. E. Norton also thanks the workshop ``Categorification in quantum topology and beyond" at the Erwin Schr\"{o}dinger Institut, Vienna, January 2019.

\section{Representation theory for types $B$ and $C$}

\subsection{Combinatorics}

\subsubsection{Partitions and symbols}
Let $m$ be a non-negative integer. A partition $\lambda$ of $m$ is a non-increasing sequence of non-negative integers
$\lambda_1 \geq \lambda_2 \geq  \cdots \geq 0$ which add up to $m$. We call $m$ the size of
$\lambda$ and we denote it by  $|\lambda|$. 
 Let $s \in \mathbb{Z}$ and $\lambda$ be a partition of $m$. The charged $\beta$-set of $\lambda$
is the set
$$ \beta_s(\lambda) := \{ \lambda_1 +s , \lambda_2 +s-1,\ldots,\lambda_i+s-i+1,\ldots \}.$$
It is a subset of $\mathbb{Z}$ which contains all $z\in\mathbb{Z}$ such that $z\leq s-\#\{\hbox{non-zero parts of }\lambda\}$.

\smallskip
A bipartition $\blambda = \lambda^1 . \lambda^2$ of $m$
consists of  a pair $(\lambda^1,\lambda^2)$ of partitions such that $|\blambda| :=  |\lambda^1|+|\lambda^2|$
equals $m$. If $\lambda^2$ (resp. $\lambda^1$) is the empty partition we will write $\blambda = \lambda^1.$ (resp.
$\lambda = .\lambda^2$).  Given $\bfs = (s_1,s_2)$ in $\mathbb{Z}^2$, a \emph{charged symbol} with charge $\bfs$ is a pair $\Lambda = (X,Y)$ where $X=\beta_{s_1}(\lambda^1)$ and $Y=\beta_{s_2}(\lambda^2)$
for some bipartition $\blambda= \lambda^1.\lambda^2$. 
The \emph{defect} of the charged symbol $\Lambda$ is $D= s_1-s_2$. We set $\Lambda^\dag = (Y,X)$. It is a symbol with charge $\bfs^* := (s_2,s_1)$ and it is associated to the bipartition $\lambda^2.\lambda^1$. Note that the defect of a symbol should not be confused with the other use of the word ``defect" arising in representation theory of finite groups, namely the defect of a block.
\smallskip

Throughout this paper we shall only be working with charged symbols with odd defect and with a specific charge. Given $t \in \mathbb{Z}$, 
let us define
\begin{equation}\label{eq:sigmat}
\bsigma_t := \left\{ \begin{array}{ll} (t,-1-t) & \text{if $t$ is even,} \\
 (-1-t,t) & \text{if $t$ is odd}. \end{array}\right.
 \end{equation}
 Note that $\bsigma_{-1-t} = \bsigma_t$. A \emph{symbol} $\Lambda= (X,Y)$ is a charged symbol with charge $\bsigma_t$ for some $t\in \mathbb{Z}$. 
 If $X=\{x_1,x_2,x_3,\ldots\}$ and $Y=\{y_1,y_2,y_3,\ldots\}$, we will represent $\Lambda$ by
$$\Lambda = \begin{pmatrix} x_1 & x_2 & x_3 & \ldots\\
y_1 & y_2 & y_3 & \ldots \\
\end{pmatrix}.$$
This convention differs from the usual convention, for example the one in \textsf{Chevie} \cite{Mi15} since we allow both positive and negative defect, and since a symbol in \textsf{Chevie} is necessarily truncated on the right whereas our symbols are infinite to the right.
 Nevertheless this will be needed to have a consistent action of the $i$-induction operators on all symbols from the various Harish-Chandra series, and to allow $n$ to grow arbitrarily large, see \S\ref{sec:branching} and especially Remark~\ref{rmk:convention}. 
 \smallskip
 
\begin{remark} The $2$-row convention for representing the charged symbol $\Lambda$ of a bipartition $\blambda$ with charge $\bfs$ is the $180^\circ$-rotation of the $2$-abacus of $\blambda$ with charge $\bfs$ as in \cite{Gerber}.
\end{remark}
\smallskip

We will sometimes find it useful to drop the notation of symbols and work with Young diagrams. The Young diagram of the bipartition $\blambda$ is the set of triples
 $$ Y(\blambda):=\{(x,y,j)\in\mathbb{N}\times \mathbb{N}\times \{1,2\}\mid 1\leq x\leq \#\{\hbox{nonzero parts of }\lambda^j\},\;1\leq y\leq \lambda^j_x \}.$$
An element $b=(x,y,j)\in Y(\blambda)$ is called a box of the Young diagram. We will draw the Young diagram of a bipartition $\blambda$ by putting the diagrams of $\lambda^1$ and $\lambda^2$ side by side. For a box $(x,y,j)$ in $Y(\blambda)$, $x$ represents the row and $y$ the column, with the convention that rows are decreasing in length from top to bottom, as illustrated below for the example $\blambda=543.21$:
\begin{center}
\ytableausetup{mathmode,boxsize=1em}
$Y(543.21)=$ \quad \ydiagram{5,4,3}
\ \ $\cdot$ \
\ydiagram{2,1}
\end{center}

\subsubsection{Cores and co-cores}
Let $d$ be a positive integer and let $\Lambda = (X,Y)$ be a symbol. A $d$-hook in the top row 
(resp. in the bottom row)
of $\Lambda$ is a pair $(x,x+d)$ such that $x + d \in X$ and $x \notin X$ (resp. $x + d \in Y$ and $x \notin Y$). 
Removing a $d$-hook in the top row amounts to changing $\Lambda$ to $((X\smallsetminus\{x+d\}) \cup \{x\},Y)$,
and similarly for the bottom row. The  $d$-core is the symbol obtained by recursively removing all possible 
$d$-hooks. Removing or adding $d$-hooks does not change the defect of the symbol. 

\smallskip

A $d$-co-hook of $\Lambda$ is a pair $(x,x+d)$ such that $x + d \in X$ and $x \notin Y$  or  $x + d \in Y$ and $x \notin X$. 
The co-hook is removed from $\Lambda$ by removing $x+d$ from $X$ and adding $x$ to $Y$, or removing $x+d$ from $Y$ and adding $x$ to $X$, and then exchanging $X$ and $Y$. Recursively removing all $d$-co-hooks yields the $d$-co-core of $\Lambda$. 

\subsubsection{Families}\label{sssec:families}
Let $\Lambda = (X,Y)$ be a symbol. The \emph{composition} $\varpi_\Lambda$ attached to the symbol is the non-increasing sequence  $\varpi_\Lambda := (\varpi_1 \geq \varpi_2 \geq \varpi_3 \geq \cdots)$ obtained by considering the union of $X$ and $Y$ as a multiset. Since $X$ and $Y$ are $\beta$-sets of some partitions, any term in the composition $\varpi_\Lambda$
occurs at most twice (and all but finitely many terms appearing do). 

\smallskip

The dominance order on compositions defines a relation on symbols. We say that two symbols $\Lambda$ and $\Lambda'$ lie in the same \emph{family} and we write $\Lambda \equiv \Lambda'$ if $\varpi_\Lambda =\varpi_{\Lambda'}$. In other words, two symbols are in the same family if their multisets of entries are the same. We write $\Lambda \lhd \Lambda'$ and we say that $\Lambda'$ dominates $\Lambda$ if $\varpi_\Lambda  \lhd \varpi_{\Lambda'}$, by which we mean $\varpi_\Lambda  \neq \varpi_{\Lambda'}$
and $ \sum_{i=1}^j \varpi_i \leq \sum_{i=1}^j \varpi_i'$
for all $j \geq 1$.  This defines a strict partial order on the set of symbols. We will write $\Lambda \unlhd \Lambda'$
if $\Lambda = \Lambda'$ or $\Lambda \lhd \Lambda'$.

\begin{example} The following four symbols 
$$\begin{pmatrix}1 & 0 & & -2 & \cdots \\
& & -1 & -2 & \cdots \end{pmatrix},
\begin{pmatrix} 1 & &  -1 &-2 &  \cdots \\
& 0 & & -2 &\cdots
\end{pmatrix},
\begin{pmatrix} & 0 & -1 & -2 &  \cdots \\
1 & & & -2 & \cdots \end{pmatrix},
\begin{pmatrix} & & & -2 & \cdots \\
1 & 0 & -1 & -2 & \cdots \end{pmatrix}$$
form a family attached to the composition $(1,0,-1,-2,-2,\ldots) $. The first three symbols have charge $(0,-1)$ and 
correspond to the bipartitions $1^2.$, $1.1$ and $.2$, whereas the fourth symbol
has charge $(-2,1)$ and corresponds to the empty bipartition. Note that we have only included symbols that have charge $\bsigma_t$ for some $t \in \mathbb{Z}$.

\end{example}

\subsection{Unipotent representations of finite reductive groups of type $B$ and $C$}

\subsubsection{Representations of finite groups}
Let $G$ be any finite group and $\bbL$ a commutative ring with unit.
We denote by $\bbL G\nmod$ the abelian category of finitely generated left $\bbL G$-modules. The set of isomorphism classes of irreducible (or simple) objects 
will be denoted by $\mathrm{Irr}_\bbL G$. We will write $K_0(\bbL G\nmod)$
for the Grothendieck group of the category $\bbL G\nmod$. 
\smallskip

Let $\ell$ be a prime number. We shall work with representations over fields of characteristic
zero and $\ell$. For that purpose we fix an $\ell$-modular system $(\bbK,\bbO,\bbk)$ where $\bbK$
is an extension of $\bbQ_\ell$, the ring of integers $\bbO$ of $\bbK$ over $\bbZ_\ell$
is a complete d.v.r and its residue field $\bbk$ has characteristic~$\ell$.
Throughout this paper we will assume that this modular system is sufficiently large,
so that the algebras $\bbK G$ and $\bbk G$ split for any finite group $G$ considered, that is, so that all irreducible representations of $G$ over $\bbK$ (resp. $\bbk$) remain irreducible over any field extension of $\bbK$ (resp. $\bbk$).
We will usually identify  $K_0(\bbK G\nmod)$ with the space of virtual characters of $G$,
and its basis $\mathrm{Irr}_\bbK G$ by the set of (ordinary) irreducible characters.
We will denote by $\langle - ; - \rangle_G$ the usual inner product on  $K_0(\bbK G\nmod)$.

\subsubsection{Finite reductive groups and Deligne--Lusztig characters}\label{sssec:frg}
We fix a non-negative integer $m$. Let $\bfG$ be a connected reductive group, quasi-simple of type $B_m$
or $C_m$, defined over the finite field $\mathbb{F}_q$. Let $F : \bfG \longrightarrow \bfG$
be the corresponding Frobenius endomorphism. The finite group $G:=\bfG^F$ is
a \emph{finite reductive group}. If $\mathbf{H}$ is any closed subgroup of $\bfG$ we will denote
by $H := \mathbf{H}^F$ the corresponding finite group.
We fix an $F$-stable maximal torus $\bfT$
of $\bfG$ contained in an $F$-stable Borel subgroup $\bfB$ of $\bfG$. 
We denote by $W_m := N_\bfG(\bfT)/\bfT$ the corresponding Weyl group, which is of type $B_m$. The choice 
of $\bfB$ defines a subset of simple reflections $S = \{s_1,s_2,\ldots,s_m\}$ on
which $F$ acts trivially. They are labeled according to the following Coxeter diagram:
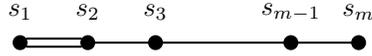
\begin{center}
\begin{tikzpicture}[scale=.9]
\node at (0,.45)[]{$s_1$};
\node at (1,.45)[]{$s_2$};
\node at (2,.45)[]{$s_3$};
\node at (4,.45)[]{$s_{m-1}$};
\node at (5,.45)[]{$s_m$};
\node at (0,0)[fill,circle,inner sep=2pt]{};
\node at (1,0)[fill,circle,inner sep=2pt]{};
\node at (2,0)[fill,circle,inner sep=2pt]{};
\node at (4,0)[fill,circle,inner sep=2pt]{};
\node at (5,0)[fill,circle,inner sep=2pt]{};
\draw[thick] plot [smooth,tension=.1]
coordinates{(.09,.06)(.91,.06)};
\draw[thick] plot [smooth,tension=.1]
coordinates{(.09,-.06)(.91,-.06)};
\draw[thick] plot [smooth,tension=.1]
coordinates{(1.1,0)(1.9,0)};
\draw[thick] plot [smooth,tension=.1]
coordinates{(2.1,0)(3.9,0)};
\draw[thick] plot [smooth,tension=.1]
coordinates{(4.1,0)(4.9,0)};
\end{tikzpicture}
\end{center}

\medskip

The $G$-conjugacy classes of $F$-stable maximal tori are parametrized 
by the conjugacy classes of $W_m$. Given $w \in W_m$ we will 
denote by $\mathbf{T}_w$ a maximal torus of type $w$.
Given $\theta$ a $\bbK$-linear character of $T_w$, 
Deligne--Lusztig defined in \cite{DeLu76} a virtual character 
$R_{T_w}^G(\theta)$ of $G$ over $\bbK$. We will write $R_w := R_{T_w}^G(1_{T_w})$
for the Deligne--Lusztig character associated to the trivial 
character of $T_w$. The irreducible constituents of the various $R_w$'s 
are the \emph{unipotent characters} of $G$.

\subsubsection{Harish-Chandra induction and restriction}
Given $I \subseteq S$, we write $W_I$ for the subgroup of $W$ generated
by $I$ and $\mathbf{P}_I = \bfB W_I \bfB$ for the corresponding standard parabolic subgroup of $\bfG$.
It has a Levi decomposition $\mathbf{P}_I = \mathbf{L}_I \ltimes \mathbf{U}_I$ where
 $\mathbf{L}_I$ is the unique Levi complement of $\mathbf{P}_I$ containing $\bfT$. 
 The Harish-Chandra induction and restriction functors are defined by
 $$ R_{L_I}^G := \bbL G/U_I \otimes_{\bbL L_I} -  \quad \text{and} \quad {^*R}_{L_I}^G := \mathrm{Hom}_G(\bbL G/U_I ,-)$$
where $\bbL$ is any of the rings $\bbK$, $\bbO$, $\bbk$. A $\bbL G$-module $V$ is said
to be \emph{cuspidal} if ${^*R}_{L_I}^G(V) = 0$ for all $I \subsetneq S$. 
When $\ell \nmid q$, the functors $({^*R}_{L_I}^G,R_{L_I}^G )$ form a biadjoint pair of
exact functors between $\bbL L_I\nmod$ and $\bbL G\nmod$.

\smallskip

When $\Lambda = \bbK$ these functors yield linear maps on characters of $G$ which we will still denote
by $R_{L_I}^G$ and ${^*R}_{L_I}^G$. When $I = \emptyset$ we have that $\mathbf{L}_I = \bfT$  is a split torus. In that case the Harish-Chandra induction and the Deligne--Lusztig map $R_T^G$ defined above coincide, which justifies our notation.

\subsubsection{Unipotent characters}\label{sssec:unipchars}
We recall here Lusztig's parametrization of unipotent characters of $G$ (see for example \cite[\S4]{Lu84}). Recall that $\bfG$ is a quasi-simple
group of type $B_m$ or $C_m$. The finite group $G$ admits a cuspidal unipotent character
if and only if $m = t^2+t$ for some $t\geq 0$. We label such a character by the charged symbol of the empty bipartition
with the charge $\bsigma_t$, see \eqref{eq:sigmat}, which is given by
$$\Lambda = \left\{ \begin{array}{ll}
\begin{pmatrix} t&t\mn1&\ldots & \ldots & \ldots\\ & & \mn1\mn t &\mn 2\mn t& \ldots\end{pmatrix} & \text{if $t$ is even}, \\[15pt]
\begin{pmatrix} & & \mn1\mn t &\mn 2\mn t& \ldots \\ t&t\mn1&\ldots & \ldots& \ldots \end{pmatrix} & \text{if $t$ is odd}. \end{array}
\right.
$$
More generally, if $t^2+t \leq m$, 
one can consider the standard Levi subgroup $L$ of $G$ of type $B_{t^2+t}$ or $C_{t^2+t}$.
Then the unipotent characters of $G$ lying in the Harish--Chandra series of the cuspidal unipotent character of $L$
correspond to bipartitions $\blambda=\lambda^1.\lambda^2$ of size $r= m-t^2-t$. In that case will we write
$\big[\lambda^1.\lambda^2\big]_{B_{t^2+t}}$ for the corresponding unipotent character, or $\big[\Lambda\big]$
if $\Lambda$ is the symbol of charge $\bsigma_t$ attached to $\blambda$. With our convention, the character
with smallest degree in the series is  $[r.]_{B_{t^2+t}}$ when $t$ is even but $[.r]_{B_{t^2+t}}$ when $t$ is odd. It agrees with the convention in \textsf{Chevie} \cite{Mi15} when $t$ is even but when $t$ is odd, the components of the bipartition must be swapped. 

\smallskip

The decomposition of the Deligne--Lusztig characters $R_w$ in terms of symbols was determined by Lusztig. Given an irreducible character $\chi$ of $W_m$ over $\bbK$ we can
form the {\em almost character }
$$R_\chi := \frac{1}{|W_m|} \sum_{w \in W_m} \chi(w) R_w.$$
Then using \cite[Thm.~4.23]{Lu84} one can compute the multiplicity $ \big\langle R_\chi;\big[\Lambda\big]\big\rangle_G$ 
of the unipotent character $[\Lambda]$ for any symbol $\Lambda$. Two examples of computations of Deligne--Lusztig characters are given in the appendix.

\subsubsection{Unipotent $\ell$-blocks}\label{sssec:unipblocks}
By an $\ell$-block $B$ of $G$ we mean a minimal $2$-sided ideal of 
the group algebra $\bbO G$. We have $B = \bbO G b$ for a unique primitive central idempotent $b$ of $\bbO G$.
We will write $\mathrm{Irr}_\bbK B$ for the ordinary irreducible characters lying in $B$, that is, those irreducible
characters $\chi \in \mathrm{Irr}_\bbK G$ such that $\chi(b) \neq 0$. An $\ell$-block $B$ is \emph{unipotent}
if it contains at least one unipotent character. In particular, the principal $\ell$-block, which is the block containing
the trivial character, is unipotent. We will denote by $\bbO G\umod$ the category of representations over the sum of the unipotent $\ell$-blocks of $G$.

\smallskip

Assume now that $\ell$ and $q$ are odd. The  $\ell$-blocks of $G$ were classified by Fong--Srinivasan in \cite{FoSr89}. There are two situations, depending on whether $\ell$ is ``linear'' or ``unitary''. Let $d$ be the multiplicative order of $q$ in $\bbF_\ell^\times$. 
\begin{itemize}
\item  If $d$ is odd, $\ell$ is said to be a \emph{linear prime} for $G$. In that case two unipotent characters $\big[\Lambda\big]$
and $\big[\Lambda'\big]$ lie in the same block if and only if the symbols $\Lambda$ and $\Lambda'$ have the same 
$d$-core. The number of $d$-hooks that must be removed to reach the $d$-core is called the  weight of the block.
\item If $d$ is even, $\ell$ is said to be a \emph{unitary prime} for $G$. Set $e := d/2$, the order of $q^2$
 in $\bbF_\ell^\times$. Then two unipotent characters $\big[\Lambda\big]$ and $\big[\Lambda'\big]$ lie in the same block if and only if the symbols $\Lambda$ and $\Lambda'$ have the same $e$-co-core. The number of $e$-co-hooks that must be removed to reach the $e$-co-core is called the  weight of the block.
\end{itemize}

We will often refer to an $\ell$-block $B$ as a $\Phi_d$-block, where $\Phi_d$ stands for the $d$-th cyclotomic polynomial.
This is justified by the fact that many of the properties of $B$ depend only on $d$ rather than on $\ell$, see \cite[Thm. 5.24]{BMM93}.
For example, if $\ell > d$ then any defect group of $B$ is isomorphic to $(\bbZ/\Phi_d(q)_\ell)^r$ where $r$ is the weight of the block.

\subsubsection{Decomposition matrix}\label{sssec:dec}
Recall that $(\bbK,\bbO,\bbk)$ is an $\ell$-modular system which is sufficiently large for $G$. Then every
$\bbK G$-module admits an integral form over $\bbO$, which can then be reduced modulo $\ell$ to 
a $\bbk G$-module. The image of that module in the Grothendieck group does not depend on 
the choice of the integral form and we obtain a linear map
$$ \mathsf{dec} : K_0(\bbK G\nmod) \longrightarrow K_0(\bbk G\nmod)$$
called the decomposition map. The decomposition matrix is the matrix
of this map in the bases $\mathrm{Irr}_\bbK G$ and $\mathrm{Irr}_\bbk G$.
It respects the block decomposition so that we can talk about the decomposition
matrix of an $\ell$-block. Dually, every projective
$\bbk G$-module $P$ lifts to a unique projective $\bbO G$-module $\widetilde P$, up to isomorphism. 
By the character of $P$ we mean the character of the $\bbK G$-module $\bbK\widetilde P$.
Brauer reciprocity states that the  decomposition matrix of $G$ is also the matrix whose 
columns are the characters of the PIMs (the projective indecomposable $\bbk G$-modules) in the basis
$\mathrm{Irr}_\bbK G$. 

\smallskip

In this paper we shall only be interested in the decomposition matrix of unipotent $\ell$-blocks. It is a reasonable restriction since any block is conjecturally Morita equivalent to some unipotent block \cite{Bro}. This was proved for a large class of non-unipotent blocks in \cite{BDR}.
We
say that a PIM is unipotent if it belongs to a unipotent $\ell$-block.  
When $\ell$ is odd, the unipotent characters form a basic set of the unipotent $\ell$-blocks \cite{GeHi91,Ge93}.
If in addition $q$ is odd, this basic set is unitriangular with respect to the order on families \cite{BDT19} (see \cite[\S4]{Spa} for the description of the order in terms of symbols). This means
that there is a labeling of the unipotent PIMs by symbols such that 
$$ \big\langle \bbK \widetilde P_\Lambda; [\Lambda']\big\rangle_G = \left\{ \begin{array}{ll} 1 & \text{if } \Lambda = \Lambda', \\
0 &  \text{if } \Lambda'  \ntrianglelefteq \Lambda.  \end{array}\right.  $$
To avoid cumbersome notation we will denote by $\Psi_{[\Lambda]}$
the unipotent part of the character of the PIM corresponding to the unipotent character $\big[\Lambda\big]$ by unitriangularity.

\smallskip

Similarly, we say that a simple $\bbk G$-module is unipotent if it belongs to a unipotent $\ell$-block. The unitriangularity of the decomposition matrix gives a natural labeling of the unipotent simple $\bbk G$-modules by unipotent characters. If $\big[\Lambda\big]$ is a unipotent character, we will denote by $S_{[\Lambda]}$ the corresponding
simple $\bbk G$-module.

\subsection{Branching rules}\label{sec:branching}
We recall and complete in this section the main result in \cite[\S6]{DVV17} on the branching rules for 
Harish-Chandra induction and restriction for unipotent representations of groups of type $B$ and $C$.
Throughout this section we will assume that $\ell$ and $q$ are odd, and that $d$, the multiplicative order of $q$ in $\mathbb{F}_\ell^\times$,
is even. In particular we have $d > 1$.

\smallskip

To avoid cumbersome notation we will work with symplectic groups but under our assumption that $\ell$ is odd, the main result of this section, Theorem \ref{thm:crystaliso}, remains valid for any group of type $B$ or~$C$.

\subsubsection{Level 2 Fock spaces}
Let $\blambda=\lambda^1.\lambda^2$ be a bipartition, and  $\bfs \in \mathbb{Z}^2$. The \emph{charged content} of a box $b = (x,y,j)$ in the Young diagram $Y(\blambda)$ is
$$\mathsf{co}^\bfs(b)=y-x+s_j.$$
Given another bipartition $\boldsymbol\mu$ and $c \in \mathbb{Z}$, we write $\boldsymbol\mu \smallsetminus \blambda = c$
if there exists a box $b$ of $\boldsymbol\mu$ with charged content $\mathsf{co}^\bfs(b) = c$ such that
the Young diagram of $\blambda$ is obtained from the Young diagram of $\boldsymbol\mu$ by removing 
the box $b$. 

\smallskip

Let $\{e_i,f_i\}_{i \in 0,\ldots,d-1}$ be the Chevalley generators of the affine Lie algebra $\widehat{\mathfrak{sl}}_d$. 
The Fock space $\mathsf{F}(\bfs)$ with charge $\bfs$ is the $\widehat{\mathfrak{sl}}_d$-module equipped with  
a $\mathbb{C}$-basis $|\blambda,\bfs\rangle$ labeled by bipartitions on which the Chevalley generators
act by
$$ e_i |\blambda,\bfs\rangle = \sum_{\begin{subarray}{c} j \equiv i \text{ mod } d \\ \blambda \smallsetminus \boldsymbol\mu=j \end{subarray}} |\boldsymbol\mu,\bfs\rangle \quad \text{and} \quad f_i |\blambda,\bfs\rangle = \sum_{\begin{subarray}{c}  j \equiv i \text{ mod } d \\ \boldsymbol\mu\smallsetminus\blambda =j \end{subarray}} |\boldsymbol\mu,\bfs\rangle.$$
Note that the action of $\widehat{\mathfrak{sl}}_d$ depends only on the class of $s_1$ and $s_2$ in $\bbZ/d$. 

\subsubsection{Order on bipartitions}
We consider here an order on bipartitions defined by Dunkl--Griffeth in \cite[\S4]{DuGr10}. Let $\blambda$
be a bipartition and let $Y(\blambda)$ be its Young diagram. If $b=(x,y,j)$ is a box in $Y(\blambda)$
we write $j(b) := j$. If $\bmu$ is another bipartition and $\bfs \in \mathbb{Z}^2$ we write $\blambda \preceq_\bfs
\bmu$ if for all $\alpha \in \mathbb{R}$ and $j = 1,2$
$$\begin{aligned}
& \, \#\{b \in Y(\blambda) \mid \mathsf{co}^\bs(b) - j(b) \frac{d}{2} >  \alpha \text{ or } \mathsf{co}^\bs(b) - j(b) \frac{d}{2} = \alpha\text{ and } j(b) \leq j\} \\
\leq& \, \#
\{b \in Y(\bmu) \mid \mathsf{co}^\bs(b) - j(b) \frac{d}{2} > \alpha\text{ or } \mathsf{co}^\bs(b) - j(b) \frac{d}{2} = \alpha \text{ and } j(b) \leq j\}. \end{aligned}$$
This is exactly the order $\leq_c$ defined in \cite{DuGr10} with $r=2$, $c_0 = d^{-1}$ and $d_0 = -d_1  = (s_1-s_2)/d+1/2$. 
Note that changing $\mathbf{s}$ to $\mathbf{s} + (s,s)$  for any $s \in \mathbb{Z}$ does not change the order. We will need the following lemma which relates the Dunkl--Griffeth order on bipartitions to the dominance order on charged symbols defined in \S\ref{sssec:families}.

\begin{lemma}\label{lem:order}
Let $\bsigma \in \mathbb{Z}^2$ and set $\bfs :=  \bsigma + \frac{1}{2} (0,d)$.
Let $\blambda$, $\bmu$ be two bipartitions, and let $\Lambda$, $\Lambda'$ be the corresponding
symbols of charge $\bsigma$. Then
$$ \boldsymbol\lambda \preceq_{\bfs}  \bmu \ \Longrightarrow \ \Lambda \lhd \Lambda' \text{ or } \Lambda \equiv \Lambda'.$$
\end{lemma}
\begin{proof}First observe that given a box $b$ in the Young diagram of $\blambda$ we have
$$ \mathsf{co}^{\bfs}(b) - j(b) \frac{d}{2} =  \mathsf{co}^{\bsigma}(b) - \frac{d}{2}.$$
We deduce that $\blambda \preceq_{\bfs} \bmu$ if and only if for all $\alpha \in \mathbb{R}$ and $j = 1,2$ we have
\begin{equation}\label{eq:dgorder}
\begin{aligned}
& \, \#\{b \in Y(\blambda) \mid \mathsf{co}^{\bsigma}(b) >  \alpha \text{ or } \mathsf{co}^{\bsigma}(b)  = \alpha \text{ and } j(b) \leq j\} \\
\leq& \, \#
\{b \in Y(\bmu) \mid \mathsf{co}^{\bsigma}(b) > \alpha\text{ or }  \mathsf{co}^{\bsigma}(b)  = \alpha \text{ and } j(b) \leq j\}. \end{aligned}
\end{equation}
Let us consider the extended Young diagram $\widetilde{Y}(\blambda)$, defined as the set of boxes $b=(x,y,j)$
with $j= 1,2$, $x \geq 1$ and $y \leq \lambda_x^j$. Unlike the usual Young diagram we do not assume $y \geq 1$ for the boxes, which means that each row is infinite on the left. The set of boxes in $\widetilde{Y}(\blambda) \smallsetminus Y(\blambda)$ does not depend on the
bipartition $\blambda$, and the number of boxes with a given content is finite, therefore one can replace $Y(\blambda)$ and $Y(\bmu)$ by $\widetilde{Y}(\blambda)$ and $\widetilde{Y}(\bmu) $ in \eqref{eq:dgorder}. 
\smallskip

Working with extended Young diagrams makes the computations easier in \eqref{eq:dgorder}. Indeed, if
$\varpi_\Lambda := (\varpi_1 \geq \varpi_2 \geq \varpi_3 \geq \cdots)$ is the composition attached to $\Lambda$,
that is, the multiset given by the union of $\beta_{\sigma_1}(\lambda^1)$ and $\beta_{\sigma_2}(\lambda^2)$ (see \S\ref{sssec:families}), then we claim that 
$$ \#\{b \in \widetilde{Y}(\blambda) \mid \mathsf{co}^{\bsigma}(b) \geq  \alpha\} = \sum_{\varpi_k \geq \alpha} (\varpi_k-\lceil\alpha\rceil).$$
To show the claim we can assume without loss of generality that $\alpha \in \mathbb{Z}$ since the contents are integers. Each row in $\widetilde{Y}(\blambda)$ corresponds to an element $\varpi_k$, and
the highest content in that row equals $\varpi_k-1$. Consequently, this row contains a box of content $\beta$ (and only one) if and only if $\varpi_k-1 \geq \beta$. Therefore
$$\begin{aligned}
 \#\{b \in \widetilde{Y}(\blambda) \mid \mathsf{co}^{\bsigma}(b) \geq  \alpha\}  & \, = \sum_{\beta \geq \alpha} 
 \#\{b \in \widetilde{Y}(\blambda) \mid \mathsf{co}^{\bsigma}(b) = \beta\} \\
 & \, = \sum_{\beta \geq \alpha} 
 \#\{k \geq 1 \mid \varpi_k-1 \geq \beta\} \\
 & \, = \sum_{\begin{subarray}{c} k, \beta \\ \varpi_k-1 \geq \beta \geq \alpha \end{subarray}} 
 1 = \sum_k \sum_{\begin{subarray}{c} \beta \\ \varpi_k-1 \geq \beta \geq \alpha \end{subarray}} 1 \\
 &\, =\sum_{\begin{subarray}{c} k\\ \varpi_k-1  \geq \alpha \end{subarray}} (\varpi_k-\alpha). 
 \end{aligned}$$
 Let $\varpi_{\Lambda'} = (\varpi_1' \geq \varpi_2' \geq \varpi_3' \geq \cdots)$ be the composition attached to
 the symbol $\Lambda'$. We deduce from \eqref{eq:dgorder} with $j = 2$ that for all $\alpha \in \mathbb{R}$
\begin{equation}\label{eq:varpi}
\sum_{\varpi_k \geq \alpha} (\varpi_k-\lceil\alpha\rceil) \leq  \sum_{\varpi_k' \geq \alpha} (\varpi_k'-\lceil\alpha\rceil).
\end{equation}
Now, if $i \geq 1$ we have
$$\sum_{k =1}^i (\varpi_k-\varpi_i') \leq 
\sum_{\varpi_k \geq \varpi_i'} (\varpi_k-\varpi_i') \leq \sum_{\varpi_k' \geq \varpi_i'} (\varpi_k'-\varpi_i') = \sum_{k =1}^i (\varpi_k'-\varpi_i')$$ 
where we used  \eqref{eq:varpi} with $\alpha = \varpi_i'$ for the second inequality. This shows that
$\sum_{k =1}^i \varpi_k \leq \sum_{k =1}^i \varpi_k'$ for all $i \geq 1$ and completes the proof.
\end{proof}

\subsubsection{Categorification of unipotent representations}
We recall here the categorification result of \cite[\S6]{DVV17}.
For $m \geq 0$ we set $G_m := \mathrm{Sp}_{2m}(q)$. Using the (unique) standard Levi subgroup
$G_{m-1}\times\mathrm{GL}_1(q)$ of $G_m$ we can form the chain 
of subgroups $\{1\} = G_0 \subset G_1 \subset \cdots \subset G_{m-1} \subset G_m \subset \cdots$.
Since $d$, the multiplicative order of $q$ in $\mathbb{F}_\ell^\times$ is even,
the group $\mathrm{GL}_1(q)$ is an $\ell'$-group and the Harish-Chandra induction and restriction
induce exact functors between $\bbk G_{m-1}\umod$ and $\bbk G_{m}\umod$
for all $m$. We can form the abelian category
$$ \bbk G_\bullet \umod := \bigoplus_{m \geq 0} \bbk G_m\umod$$
of  the modules over all the unipotent $\ell$-blocks of the various groups $G_m$.
We will denote by $F$ and $E$ the endofunctors of this category induced
by Harish-Chandra induction and restriction respectively.

\smallskip

Since $\ell$ is odd, the unipotent characters form a basic set for the unipotent blocks \cite{GeHi91,Ge93}.
In particular $K_0(\bbk G_\bullet \umod)$ has a $\bbZ$-basis given by the image of
the unipotent characters under the decomposition map, see \S\ref{sssec:dec}. Recall
from \S\ref{sssec:unipchars} that the unipotent characters in a $B_{t^2+t}$-series are labeled by bipartitions,
or equivalently by symbols of charge $\bsigma_t$. Therefore they are in bijection
with the standard basis of any Fock space $\mathsf{F}(\bfs)$. For our purpose
we will consider the charges $\bfs_t$ defined by
\begin{equation}\label{eq:st}
\bfs_t :=  \bsigma_t + \frac{1}{2} (0,d) = \left\{ \begin{array}{ll} (t,-1-t + \frac{d}{2}) & \text{if $t$ is even,} \\[3pt]
 (-1-t,t + \frac{d}{2}) & \text{if $t$ is odd}, \end{array}\right.
 \end{equation}
 for all $t \geq 0$.  The previous discussion shows that there is an isomorphism of $\mathbb{C}$-vector spaces 
\begin{equation}\label{eq:isofock}
\begin{array}{rcl} \displaystyle\bigoplus_{t \geq 0} \mathsf{F}(\bfs_t)   & \mathop{\longrightarrow}\limits^\sim  & \mathbb{C} \otimes_\mathbb{Z} K_0(\bbk G_\bullet\umod) \\[6pt] |\blambda,\bfs_t\rangle & \longmapsto & \mathsf{dec}\big[\blambda\big]_{B_{t^2+t}} \end{array}
 \end{equation}
which sends an element of the standard basis to the image under the decomposition map of the corresponding unipotent character. One of the main results in \cite[\S6]{DVV17} is
a categorification of \eqref{eq:isofock}. It gives, for every $i = 0,\ldots,d-1$ a construction of a biadjoint pair of exact endofunctors $(F_i,E_i)$ of $\bbk G_\bullet \umod$ 
such that
$$F = \bigoplus_{i = 0}^{d-1} F_i \quad \text{and} \quad E = \bigoplus_{i=0}^{d-1} E_i$$
called {\em $i$-induction} and {\em $i$-restriction functors}, which induce an action of 
$\widehat{\mathfrak{sl}}_d$ on $K_0(\bbk G_\bullet\umod)$ making 
\eqref{eq:isofock} an isomorphism of $\widehat{\mathfrak{sl}}_d$-modules. 

\smallskip

\begin{remark}\label{rmk:convention}
In \cite[(6.3)]{DVV17} the authors used the charge $\bfs_t^*$ instead of $\bfs_t$ when $t$ is odd. This does not affect the categorification
result since the Fock spaces $\mathsf{F}(\bfs_t^*)$ and $\mathsf{F}(\bfs_t)$ are clearly isomorphic,
but it explains the discrepancy in our notation for unipotent characters with the one in $\mathsf{Chevie}$.
With our convention, under the isomorphism \eqref{eq:isofock}, the action of $f_i$ on a symbol
$$\Lambda = \begin{pmatrix} x_1 & x_2 & x_3 & \ldots\\
y_1 & y_2 & y_3 & \ldots \\
\end{pmatrix}$$ 
is given by increasing by $1$ any $x_j$ equal to $i$ modulo $d$ or any $y_j$ equal to $i+d/2$ modulo $d$, when possible.
\end{remark}

\subsubsection{Crystal graph and branching rules}\label{sssec:branchingrules}
Fix a charge $\bs=(s_1,s_2)\in\mathbb{Z}^2$, let $d\in\mathbb{Z}_{\geq 2}$ (here we do not require $d$ to be even), and consider the Fock space $\mathsf{F}(\bs)$. Let $|\blambda,\bs\rangle$ be a charged bipartition in $\mathsf{F}(\bs)$. A box $b$ of $\blambda$ is {\em removable} if $\bmu:=\blambda\setminus b$ is a bipartition. Then $b$ is called an {\em addable} box of the bipartition $\bmu$. For each $i\in\mathbb{Z}/d\mathbb{Z}$, define the {\em $i$-word} of $|\blambda,\bs\rangle$ as follows: list all the addable and removable boxes $b$ of $|\blambda,\bs\rangle$ such that $\mathsf{co}^\bs(b)\equiv i \mod d$ in increasing order from left to right according to  their value in $\bbZ$, 
with the convention that if $\mathsf{co}^\bs(b)=\mathsf{co}^\bs(b')$ and $b\in\lambda^1,b'\in\lambda^2$, then $b'$ is smaller than $b$. Now replace each addable box in the list by the symbol $+$ and each removable box in the list by the symbol $-$. The resulting string of pluses and minuses is called the $i$-word of $|\blambda,\bs\rangle$. The {\em reduced $i$-word} of $|\blambda,\bs\rangle$ is then found from the $i$-word  by recursively canceling all adjacent pairs $(-+)$. The reduced $i$-word is of the form $(+)^a(-)^b$ for some $a,b\in\mathbb{Z}_{\geq 0}$. The {\em Kashiwara operator $\tilde{f}_i$} adds the addable $i$-box corresponding to the rightmost $+$ in the reduced $i$-word of $|\blambda,\bs\rangle$, or if there is no $+$ in the reduced $i$-word then it acts by $0$. Likewise, the Kashiwara operator $\tilde{e}_i$ removes the removable $i$-box corresponding to the leftmost $-$ in the reduced $i$-word of $|\blambda,\bs\rangle$, or if there is no $-$ in the reduced $i$-word then it acts by $0$. The directed graph with vertices all bipartitions and ($\mathbb{Z}/d\mathbb{Z}$-colored) edges $\blambda\stackrel{i}{\rightarrow}\bmu$ if and only if $\bmu=\tilde{f}_i(\blambda)$, $i\in\mathbb{Z}/d\mathbb{Z}$, is called the {\em $\widehat{\mathfrak{sl}}_d$-crystal} on $\mathsf{F}(\bs)$ \cite[Section 3]{JMMO}, \cite[Theorem 2.8]{FLOTW}.
\smallskip

\begin{example}
Let $\blambda=421^4.4^232^21^3$, $\bs=(1,0)$, and $d=4$.
\[|\blambda,\bs\rangle=
\begin{ytableau}
1&2&3&4\\
0&1\\
\hbox{-}1\\
\hbox{-}2\\
\hbox{-}3\\
\hbox{-}4\\
\end{ytableau}
\ \ \cdot \ \,
\begin{ytableau}
0&1&2&3\\
\hbox{-}1&0&1&2\\
\hbox{-}2&\hbox{-}1&0\\
\hbox{-}3&\hbox{-}2\\
\hbox{-}4&\hbox{-}3\\
\hbox{-}5\\
\hbox{-}6\\
\hbox{-}7\\
\end{ytableau}
\]
Let us find the $0$-word of $|\blambda,\bs\rangle$. The addable $0$-boxes $(x,y,j)$, where $x$ is the row, $y$ is the column, and $j$ is the component, are: $(9,1,2),(6,2,2),(3,2,1), (1,5,2)$. The removable $0$-boxes are: $(6,1,1),(3,3,2),(1,4,1)$. Ordering them we obtain the $0$-word:
\[
\begin{array}{ccccccc}
(9,1,2)&(6,2,2)&(6,1,1)&(3,3,2)&(3,2,1)&(1,5,2)&(1,4,1)\\
+&+&-&-&+&+&-
\end{array}
\]
Iterating cancellations of all adjacent $(-+)$, we obtain the reduced $0$-word:
\[
\begin{array}{ccc}
(9,1,2)&(6,2,2)&(1,4,1)\\
+&+&-
\end{array}
\]
Thus $\tilde{f}_0$ adds the box $(6,2,2)$ to $|\blambda,\bs\rangle$, and $\tilde{e}_0$ removes the box $(1,4,1)$ from $|\blambda,\bs\rangle$. That is, $\tilde{f}_0(\blambda)=421^4.4^232^31^2$ and $\tilde{e}_0(\blambda)=321^4.4^232^21^3$.
\end{example}
\smallskip

We will often work with symbols instead of Young diagrams, so it is useful to describe how the operators $\tilde{f}_i$ and $\tilde{e}_i$ act on symbols. Let $[\blambda]_{B_{t^2+t}}$ be a bipartition in the $B_{t^2+t}$ series, and let  $\Lambda=(X,Y)$ be its symbol with charge $\bsigma_t$. Recall that this does not depend on $d$. Instead of changing the symbol to depend on $d$, we define the action so that it depends on $d$. 

\smallskip

Let $d\in2\mathbb{N}$ with $d \geq 2$.
An addable $i$-box of $\lambda^1$ is $x\in X$ such that $x+1\notin X$ and $x\equiv i \mod d$. A removable $i$-box of $\lambda^1$ is $x\notin X$ such that $x+1\in X$ and $x\equiv i \mod d$. The next condition records the dependence on $d$. An addable $i$-box of $\lambda^2$ is $y\in Y$ such that $y+1\notin Y$ and $y+\frac{d}{2}\equiv i \mod d$. A removable $i$-box of $\lambda^2 $ is $y\notin Y$ such that $y+1\in Y$ and $y+\frac{d}{2}\equiv i \mod d$. We then order all addable and removable $i$-boxes of $\blambda$
by the following rule:
\begin{itemize}
\item $x_1$ is less than $x_2$ if and only if $x_1<x_2$ in $\mathbb{Z}$,
\item $y_1$ is less than $y_2$ if and only if $y_1<y_2$ in $\mathbb{Z}$,
\item $x$ is less than $y$  if and only if $y+\frac{d}{2}>x$ in $\mathbb{Z}$,
\item $y$ is less than $x$ if and only if $x\geq y+\frac{d}{2}$ in $\mathbb{Z}$,
\end{itemize}
for all addable and removable boxes $x,x_1,x_2$ of $\lambda^1$ and $y,y_1,y_2$ of $\lambda^2$.
We then form the $i$-word of $\blambda$ by listing its addable and removable $i$-boxes in increasing order as just defined. The cancellation rule is the same as before, and we obtain the good addable $i$-box and the good removable $i$-box of $\blambda $ (if they exist) as described before.
 
 \begin{example}
 Take $\blambda=31^2.2$ and consider it as a bipartition in the principal series, so that $t=0$. We have $\bsigma_0=(0,-1)$ and thus the symbol associated to $[\blambda]$ is
 $$\Lambda=\begin{pmatrix}X\\Y\end{pmatrix}=
 \begin{pmatrix}
 3&&&0&-1&&-3&-4&\ldots\\
 &&1&&&-2&-3&-4&\ldots
 \end{pmatrix}
 $$
 Now suppose that $d=6$. Let us calculate the action of $\tilde{f}_3$ and $\tilde{e}_3$ on $\Lambda$. The addable $3$-boxes of $\blambda$ are given by $3,-3\in X$, and the removable $3$-boxes of $\blambda$ are given by $0\notin Y$. Ordering them according to our rule, the $3$-word of $\blambda$ is:
\[
\begin{array}{ccc}
+ & - & +\\
-3\in X & 0\notin Y & 3\in X
\end{array}
\]
We then cancel the occurrence of $(-+)$, yielding $+$ as the reduced $3$-word and $-3\in X$ as the good addable $3$-box. There is no good removable $3$-box, so $\tilde{e}_3\Lambda=0$. Adding the good addable $3$-box  to $\blambda$ yields:
\[
\tilde{f}_3\Lambda=
 \begin{pmatrix}
 3&&&0&-1&-2&&-4&\ldots\\
 &&1&&&-2&-3&-4&\ldots
 \end{pmatrix}.
\]
 \end{example}
\smallskip

Now we recall how the combinatorics of crystals relates to the representation theory of unipotent blocks. We can define the \emph{colored branching graph} whose vertices are labeled by the
unipotent simple $\bbk G_m$-modules for all $m \geq 0$ and arrows are given by $T \xrightarrow{i} T'$ if $T'$ appears in the head of $F_i(T)$,
or equivalently if $T$ appears in the socle of $E_i(T')$. 

\smallskip

\begin{theorem} \label{thm:crystaliso}
The map 
$$|\blambda,\bfs_t\rangle \mapsto S_{[\blambda]_{B_{t^2+t}}}$$
induces an isomorphism between the union of the crystal graphs of $\mathsf{F}(\bfs_t)$ for $t\geq 0$ and
the colored branching graph of Harish-Chandra induction and restriction.
\end{theorem}

\begin{proof}The proof follows the arguments given in the proof of \cite[Thm.~4.37]{DVV19}. Let $t \geq 0$.
There is a perfect basis of the Fock space $\mathsf{F}(\bfs_t)$ coming from an upper global basis of a quantum
deformation of $\mathsf{F}(\bfs_t)$ defined by Uglov in \cite{Uglov}. We denote this basis by $b^\vee\left(|\blambda,\bfs_t\rangle\right)$.
It is unitriangular in the standard basis, with respect to the order $\preceq_{\mathbf{s}_t}$ defined in \S\ref{sec:branching}. In other words we have 
\begin{equation}\label{eq:buni}
 b^\vee\left(|\blambda,\bfs_t\rangle\right) \in |\blambda,\bfs_t\rangle + \sum_{\blambda \prec_{{\bfs}_t} \boldsymbol\mu} \mathbb{C} |\boldsymbol\mu,\bfs_t\rangle.
 \end{equation}
On the other hand, since the decomposition matrix is unitriangular, we have, for every unipotent character
$\big[\Lambda\big]$ attached to a symbol $\Lambda$
\begin{equation}\label{eq:suni}
\mathsf{dec}\big(\big[\Lambda\big]\big) \in S_{[\Lambda]} + \sum_{\Lambda \lhd \Lambda'} \mathbb{Z} S_{[\Lambda']} \end{equation}
in $K_0(\bbk G_\bullet\umod)$. Now let $\psi$ be the isomorphism in \eqref{eq:isofock} sending $|\blambda,\bs_t\rangle$ to $\mathsf{dec}\big(\big[\Lambda\big]\big)$ where $\Lambda$ is the symbol of charge $\bsigma_t$ attached to $\blambda$. Lemma~\ref{lem:order} tells us that for all $\bmu$ such that $\blambda\prec_{{\bs}_t}\bmu$, it holds that $\Lambda\lhd \Lambda'$ or $\Lambda \equiv \Lambda'$, where $\Lambda'$ is the symbol of charge $\bsigma_t$ attached to $\bmu$.
Since 
$
\psi\left( |\bmu,\bs_t\rangle \right)\in S_{[\Lambda']}+\sum\limits_{\Lambda'\lhd\Lambda''}\mathbb{C}S_{[\Lambda'']}
$ by \eqref{eq:suni} 
we deduce that
$$
\psi\left( |\bmu,\bs_t\rangle \right)\in S_{[\Lambda']}+\sum\limits_{\Lambda\lhd\Lambda''}\mathbb{C}S_{[\Lambda'']}.
$$
This together with \eqref{eq:buni} gives
\begin{equation}\label{eq:ssuni} \psi\big(b^\vee\left(|\blambda,\bfs_t\rangle\right) \big)  \in S_{[\Lambda]} + \sum_{\Lambda \prec \Lambda'} \mathbb{C} S_{[\Lambda']}
\end{equation}
where $\Lambda\prec\Lambda'$ is the transitive closure of the relation
$$\Lambda\lhd\Lambda', \hbox{ or }\Lambda'\leftrightarrow\bmu\hbox{ with }\blambda\prec_{{\bs}_t}\bmu.
$$
If we define $\varphi$ to be the bijection that sends $S_{[\Lambda]}$ to $\psi\big(b^\vee\left(|\blambda,\bfs_t\rangle\right) \big)$, two perfect bases of $\bigoplus_{t \geq 0} \mathsf{F}(\bfs_t)$, then \eqref{eq:ssuni} shows that $\varphi$ satifies the assumptions of \cite[Prop. 1.14]{DVV19} with respect to the order $\prec$ and therefore induces a crystal isomorphism.
\end{proof}


\section{The decomposition matrix of the principal $\Phi_{2n}$-block of $\mathrm{Sp}_{4n}(q)$ and $\mathrm{SO}_{4n+1}(q)$}

We fix an integer $n \geq 0$. Throughout this section $G = G_{2n}$ will denote one of the finite groups $\mathrm{SO}_{4n+1}(q)$ or $\mathrm{Sp}_{4n}(q)$ for $q$ a power of an odd prime. We are interested in the decomposition matrix of the principal $\ell$-block of $G$  when the order of $q$ in $\mathbb{F}_\ell^\times$ equals $2n$.
We will often refer to this block as the principal $\Phi_{2n}$-block of $G$ (see \S\ref{sssec:unipblocks}). 

\subsection{Unipotent characters in the principal $\Phi_{2n}$-block}\label{ssec:2nblock}
Let $B$ be the principal $\Phi_{2n}$-block of $G$ and $b$ be the corresponding block idempotent. 
By \S\ref{sssec:unipblocks} the block has weight $2$ and the unipotent characters lying in $B$ are labeled by symbols of rank $2n$ and $n$-co-core equal to the symbol
\begin{equation}
\label{eq:trivialsymb}
\begin{pmatrix} 0&-1&-2&\ldots\\ &-1&-2&\ldots\end{pmatrix}.
\end{equation}
Consequently one obtains such symbols by adding two $n$-co-hooks to the symbol \eqref{eq:trivialsymb}.
Depending on which rows the co-hooks are inserted in, one gets symbols of defect $1$, $-3$ or $5$ which label unipotent characters in the principal series, the $B_2$-series, or the $B_6$-series of $G$ respectively, see \S\ref{sssec:unipchars}. 

\begin{enumerate}
\item \textit{Principal series}. There are three families of symbols of defect $1$ obtained by adding two $n$-co-hooks to the
symbol \eqref{eq:trivialsymb}.
\begin{itemize}\item For $0\leq j\leq n$ and $0<i<n$, the symbol $$\begin{pmatrix}n\mn i&&0&-1 & \ldots&\widehat{\mn j}&\ldots \\ n\mn j&&&-1& \ldots&\widehat{\mn i}&\ldots \end{pmatrix}$$ corresponding to the principal series character $\big[(n\mn i)1^j.(n\mn j\pl1)1^{i\mn1}\big]$. There are $n^2-1$ such symbols.
\item For $0\leq j<2n$, the symbol $$\begin{pmatrix}2n\mn j&&0&-1&\ldots&\widehat{\mn j}&\ldots \\ &&&-1&\ldots &\ldots&\ldots\end{pmatrix}$$ corresponding to the principal series character $\big[(2n\mn j)1^j.\big]$. There are $2n$ such symbols.
\item For $0<i\leq 2n$, the symbol $$\begin{pmatrix}&&0&-1&\ldots&\ldots&\ldots\\ 2n\mn i&&&-1&\ldots&\widehat{\mn i}&\ldots\end{pmatrix}$$ corresponding to the principal series character $\big[.(2n\mn i\pl1)1^{i-1}\big]$.
Again, there are $2n$ such symbols.\\
\end{itemize}

\item \textit{$B_2$-series}. The symbols of defect $-3$ are given by
$$\begin{pmatrix}&&0&-1&\ldots&\widehat{\mn i}&\ldots&\widehat{\mn j}&\ldots\\ n\mn i&n\mn j&&-1&\ldots&\ldots&\ldots&\ldots&\ldots
\end{pmatrix}$$ 
for $0\leq i<j\leq n$. There are $n(n+1)/2$ such symbols. They correspond to the unipotent characters $\big[2^i1^{j\mn i\mn1}.(n\mn i\mn1)(n\mn j)\big]_{B_2}$, all of which lie in the $B_2$-series.

In terms of bipartitions and their Young diagrams, the unipotent characters $[\blambda]_{B_2}$ in the principal $\Phi_{2n}$-block are labeled by $\blambda$ such that $\lambda^1$ fits inside the rectangle $2^{n-1}$, and $\lambda^2$ is the reflection across a diagonal line of slope $1$ of the skew shape $2^{n-1}\setminus \lambda^1$ (in particular, $\lambda^2$ fits in the rectangle $(n-1)^2$). This is best illustrated with a picture (here, for $n=7$, we have shaded $\blambda=2^31^2.31$):
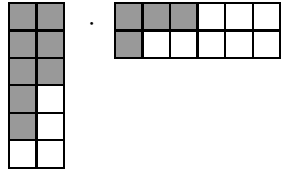
\[
\ydiagram[*(MediumGray)]{2,2,2,1,1}*[*(white)]{2,2,2,2,2,2}\ \ \cdot \ \,
\ydiagram[*(MediumGray)]{3,1}*[*(white)]{6,6}
\]

\item \textit{$B_6$-series}.   The symbols of defect $5$ are given by 
$$\begin{pmatrix}n\mn i&n\mn j&&0&-1&\ldots&\ldots&\ldots&\ldots&\ldots \\
&&&&-1&\ldots&\widehat{\mn i}&\ldots&\widehat{\mn j}&\ldots
\end{pmatrix}$$ 
for $0<i<j<n$. There are $(n-1)(n-2)/2$ such symbols. They correspond to the unipotent characters $\big[(n\mn i\mn2)(n\mn j\mn1).2^{i\mn1}1^{j\mn i\mn1}\big]_{B_6}$, all lying in the $B_6$-series.

In terms of bipartitions and their Young diagrams, the unipotent characters $[\blambda]_{B_6}$ in the principal $\Phi_{2n}$-block are labeled by $\blambda$ such that $\lambda^2$ fits inside the rectangle $2^{n-3}$, and $\lambda^1$ is the reflection across a diagonal line of slope $1$ of the skew shape $2^{n-3}\setminus \lambda^2$ (in particular, $\lambda^1$ fits in the rectangle $(n-3)^2$). Again, this is best illustrated with a picture (here, for $n=7$, we have shaded $\blambda=41.1^3$):
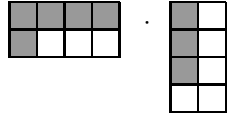
\[
\ydiagram[*(MediumGray)]{4,1}*[*(white)]{4,4}\ \ \cdot \ \,
\ydiagram[*(MediumGray)]{1,1,1}*[*(white)]{2,2,2,2}
\]

\end{enumerate}
Consequently there are $2n^2+3n$ unipotent characters in the principal $\Phi_{2n}$-block. By \cite{BMM93}, the number of unipotent characters in the principal $\Phi_{2n}$-block may also be counted by the number of complex irreducible representations of the complex reflection group $G(2n,1,2)$.

\subsection{Induced columns} In this section, we describe the columns of the decomposition matrix of the principal $\Phi_{2n}$-block obtained by Harish-Chandra induction from proper Levi subgroups.  These account for all but four columns of the decomposition matrix.   
In all cases, we find that each non-cuspidal column in the principal $\Phi_{2n}$-block is the character of an induced PIM, cut to the principal block. 

\smallskip

Since a group of type $A_r$ has no cuspidal unipotent module (over $\mathbb{k}$) unless $r=0$ or $r \geq 2n-1$, a cuspidal support of a simple unipotent $\mathbb{k}G$-module corresponds to either a Levi subgroup of type $A_{2n-1}$ or a Levi subgroup of type $B\hskip-1mmC_{r}$ for some $r \leq 2n$. In particular, if $\Lambda$ is a symbol of defect $ D \neq 1$ then $S_{[\Lambda]}$ is a cuspidal $\mathbb{k}G$-module whenever $\widetilde{e}_i \Lambda = 0$ for all $i$.

\subsubsection{The $B_2$-series submatrix}\label{B2 dec} 

\begin{theorem}\label{thm:B2 submatrix} Let $\blambda$ be a bipartition such that $[\blambda]_{B_2}$ belongs to the principal $\Phi_{2n}$-block. If $\blambda=2^{n-1}.$ then $S_{[\blambda]_{B_2}}$ is cuspidal and so $\Psi_{[\blambda]_{B_2}}$ cannot be obtained by Harish-Chandra induction. Otherwise, using truncated induction we obtain the unipotent parts of projective indecomposable characters in the $B_2$-series as follows:
\begin{itemize}
\item If $\blambda=2^k1^{n\mn 1\mn k}.(n\mn1\mn k)$ for some $0\leq k<n\mn1$, then  
$$\Psi_{[\blambda]_{B_2}}=[\blambda]_{B_2}+[2^{k\pl 1}1^{n\mn 2\mn k}.(n\mn2\mn k)]_{B_2}.$$
\item If $\blambda=2^{n\mn 2}.1^2$ then 
$$\Psi_{[\blambda]_{B_2}}=[\blambda]_{B_2}+[2^{n\mn 2}1.1]_{B_2}.$$
\item If $\blambda=2^{n\mn 1\mn k}.k^2$ for some $2\leq k\leq n\mn1$ then 
$$\Psi_{[\blambda]_{B_2}}=[\blambda]_{B_2} + [2^{n\mn 1\mn k}1.k(k\mn1)]_{B_2} + [2^{n\mn k}1.(k\mn 1)(k\mn 2)]_{B_2} + [2^{n\pl 1\mn k}.(k\mn 2)^2]_{B_2}.$$
\item If $\blambda=2^i1^{j\mn i\mn1}.(n\mn i\mn1)(n\mn j)$ with $0<i+1<j<n$ then
\begin{align*} 
\Psi_{[\blambda]_{B_2}}&=[\blambda]_{B_2} + [2^{i\pl1}1^{j\mn i \mn 2}.(n\mn i\mn2)(n\mn j)]_{B_2} + [2^i1^{j\mn i}.(n\mn i\mn 1)(n\mn j\mn 1)]_{B_2} \\&\quad + [2^{i\pl 1}1^{j\mn i\mn 1}.(n\mn i\mn 2)(n\mn j\mn 1)]_{B_2}.
\end{align*}
\end{itemize}
\end{theorem}

\begin{proof}
If $\blambda=2^{n-1}.$ then $\tilde{e}_i\blambda=0$ for all $i\in\mathbb{Z}/2n\mathbb{Z}$. Since $S_{[\blambda]_{B_2}}$ does not have cuspidal support on a type $A$ parabolic, it follows that $S_{[\blambda]_{B_2}}$ is cuspidal. For the unipotent part of the characters of the projective covers of non-cuspidal simple representations in the $B_2$-series, 
we will find $\Psi_{[\blambda]_{B_2}}$ by taking the $i$-induction of $\Psi_{[\bmu]_{B_2}}$ for some $i\in\mathbb{Z}/2n\mathbb{Z}$ such that $\tilde{f}_i\bmu=\blambda$. We find that in this case, $f_i\Psi_{[\bmu]_{B_2}}$ is indecomposable and equals $\Psi_{[\blambda]_{B_2}}$. Note that the $\ell$-blocks of $G_{2n-1}$ have either trivial or cyclic defect groups, so that the PIMs can be easily obtained from \cite{FoSr90}.
\smallskip

Let $\blambda=2^k1^{n\mn 1\mn k}.(n\mn1\mn k)$ for some $0\leq k<n-1$. Set $\bmu=2^k1^{n\mn1\mn k}.(n\mn2\mn k)$. Then $\blambda=\tilde{f}_{2n\mn k\mn1}\bmu$. We have $\Psi_{[\bmu]_{B_2}}=[\bmu]_{B_2}$ by \cite{FoSr90}. (Note that this implies the cuspidal support of $S_{[\bmu]_{B_2}}$, and thus also of $S_{[\blambda]_{B_2}}$, is different from $[B_2]$, and is labeled by some non-empty bipartition in the $B_2$-series, see \cite{FoSr90}.) Then $\Psi_{[\blambda]_{B_2}}=\Psi_{[\tilde{f}_{2n\mn k\mn1}\bmu]_{B_2}}$ is a summand of 
 $$f_{2n\mn k\mn1}\Psi_{[\bmu]_{B_2}}=f_{2n\mn k\mn1}[\bmu]_{B_2}=[\blambda]_{B_2}+[2^{k\pl1}1^{n\mn2\mn k}.(n\mn2\mn k)]_{B_2}.$$
 Thus either $f_{2n\mn k\mn1}\Psi_{[\bmu]_{B_2}}$ is the unipotent part of an indecomposable projective character and equals $\Psi_{[\blambda]_{B_2}}$, or it is the sum of the unipotent parts of two indecomposable  projective characters, each of which would have to contain a single unipotent character. 
If $k=n-2$ then $\Psi_{[2^{k\pl1}1^{n\mn 2\mn k}.(n\mn2\mn k)]_{B_2}}=\Psi_{[2^{n\mn1}.]_{B_2}}$ does not appear as a summand of any induced projective character by cuspidality. Therefore $\Psi_{[\blambda]_{B_2}}=[\blambda]_{B_2}+[2^{k\pl1}1^{n\mn 2\mn k}.(n\mn2\mn k)]_{B_2}$ when $k=n-2$. Then by downwards induction on $k$, $f_{2n\mn1\mn k}\Psi_{[\bmu]_{B_2}}=\Psi_{[\blambda]_{B_2}}$ for all $0\leq k<n-2$ as well, and in particular, the desired character formula holds.

\smallskip

Next, let $\blambda=2^{n-2}.1^2$. Set $\bmu=2^{n-2}.1$. Then $\blambda=\tilde{f}_{n}\bmu$. We have  $\Psi_{[\bmu]_{B_2}}=[\bmu]_{B_2}$ by \cite{FoSr90}.  Therefore $\Psi_{[\blambda]_{B_2}}=\Psi_{[\tilde{f}_{n}\bmu]_{B_2}}$ is a summand of 
$$f_{n}\Psi_{[\bmu]_{B_2}}=f_{n}[\bmu]_{B_2}=[\blambda]_{B_2}+[2^{n-2}1.1]_{B_2},$$
thus either $\Psi_{[\blambda]_{B_2}}=[\blambda]_{B_2}$ and $[2^{n-2}1.1]_{B_2}=\Psi_{[2^{n-2}1.1]_{B_2}}$, or $\Psi_{[\blambda]_{B_2}}=f_{n}\Psi_{[\bmu]_{B_2}}$. We showed in the previous paragraph that $\Psi_{[2^{n-2}1.1]_{B_2}}=[2^{n-2}1.1]_{B_2}+[2^{n-1}.]_{B_2}$. Therefore  $\Psi_{[\blambda]_{B_2}}=f_{n}\Psi_{[\bmu]_{B_2}}$ and the desired character formula holds.

\smallskip

In the case that $\blambda=2^{n\mn1\mn k}.k^2$ for $2\leq k\leq n-1$, we have $\blambda=\tilde{f}_{n\pl k\mn1} 2^{n\mn 1\mn k}.k(k\mn 1)$. In the case that $\blambda=2^i1^{j\mn i\mn 1}.(n\mn i\mn 1)(n\mn j)$ for $1< j-i<n$, we have $\blambda=\tilde{f}_{2n\mn i\mn 1}2^i1^{j\mn i\mn 1}.(n\mn i\mn 2)(n\mn j)$.
 By \cite{FoSr90} the following characters are the unipotent part of PIMs of $G_{2n-1}(q)$:
\begin{align*}
&\Psi_{[2^{n\mn 1\mn k}.k(k\mn 1)]_{B_2}}=[2^{n\mn k\mn 1}.k(k\mn 1)]_{B_2}+[2^{n\mn k}1.(k\mn 2)^2]_{B_2},\\
&\Psi_{[ 2^i1^{j\mn i\mn 1}.(n\mn i\mn 2)(n\mn j) ]_{B_2}}=[ 2^i1^{j\mn i\mn 1}.(n\mn i\mn 2)(n\mn j) ]_{B_2}+[2^i1^{j\mn i}.(n\mn i\mn 2)(n\mn j\mn1)]_{B_2}.
\end{align*}
 In either case, the cuspidal support of $S_{[\blambda]_{B_2}}$ is then equal to $[B_2]$, see \cite{FoSr90} (and the source vertex of $\blambda $ in the $\widehat{\mathfrak{sl}}_{2n}$-crystal is the empty bipartition). This means that if $\bnu$ is one of the bipartitions of the first two types listed in the theorem or  $\bnu=2^{n-1}.$, then $\Psi_{[\bnu]_{B_2}}$ cannot appear as a summand of $f_{n\pl k\mn 1}\Psi_{[2^{n\mn 1\mn k}.k(k\mn 1)]_{B_2}}$ or $f_{2n\mn i\mn 1}\Psi_{[2^i1^{j\mn i\mn 1}.(n\mn i\mn 2)(n\mn j)]_{B_2}}$. 

\smallskip

 We compute that 
\begin{align*}
&f_{n\pl k\mn 1}\Psi_{[2^{n\mn 1\mn k}.k(k\mn 1)]_{B_2}}=[2^{n\mn 1\mn k}.k^2]_{B_2} + [2^{n\mn 1\mn k}1.(k)(k\mn 1)]_{B_2} + [2^{n\mn k}1.(k\mn 1)(k\mn 2)]_{B_2} \\&\qquad\qquad\qquad \qquad\qquad + [2^{n\pl 1\mn k}.(k\mn 2)^2]_{B_2},\\
&f_{2n\mn i\mn 1}\Psi_{[2^i1^{j\mn i\mn 1}.(n\mn i\mn 2)(n\mn j)]_{B_2}}=[2^i1^{j\mn i\mn 1}.(n\mn i\mn 1)(n\mn j)]_{B_2} + [2^{i\pl 1}1^{j\mn i \mn 2}.(n\mn i\mn 2)(n\mn j)]_{B_2}\\&\qquad\qquad\qquad \qquad\qquad + [2^i1^{j\mn i}.(n\mn i\mn 1)(n\mn j\mn 1)]_{B_2} + [2^{i\pl 1}1^{j\mn i\mn 1}.(n\mn i\mn 2)(n\mn j\mn 1)]_{B_2}.
\end{align*}
By induction on the number of boxes in $\lambda^2$, together with the fact that no $\Psi_{[\bnu]_{B_2}}$ as above can be a summand of these induced projective characters, it follows that for none of the bipartitions $\brho$ appearing on the right-hand side of these formulas can $\Psi_{[\brho]_{B_2}}$ be a summand except for $\brho=\blambda$. 
That is, $f_{n\pl k\mn 1}\Psi_{[2^{n\mn 1\mn k}.k(k\mn 1)]_{B_2}}=\Psi_{[2^{n\mn 1\mn k}.k^2]_{B_2}}$ and $f_{2n\mn i\mn 1}\Psi_{[2^i1^{j\mn i\mn 1}.(n\mn i\mn 2)(n\mn j)]_{B_2}}=\Psi_{[2^i1^{j\mn i\mn 1}.(n\mn i\mn 1)(n\mn j)]_{B_2}}$.
We thus get the desired character formulas for $\Psi_{[\blambda]_{B_2}}$ in the final two cases listed in the theorem. This concludes the proof.
\end{proof}

\subsubsection{The $B_6$-series submatrix}\label{B6 dec} The proof of Theorem \ref{thm:B6 submatrix} is identical to the proof of Theorem \ref{thm:B2 submatrix}, except that the roles of $\lambda^1$ and $\lambda^2$ are switched.

\begin{theorem}\label{thm:B6 submatrix}
Let $\blambda$ be a bipartition such that $[\blambda]_{B_6}$ belongs to the principal $\Phi_{2n}$-block. If $\blambda=.2^{n\mn 3}$ then $S_{[\blambda]_{B_6}}$ is cuspidal and so $\Psi_{[\blambda]_{B_6}}$ cannot be obtained by Harish-Chandra induction. Otherwise, using truncated induction we obtain the unipotent parts of projective indecomposable characters in the $B_6$-series as follows:
\begin{itemize}
\item If $\blambda=(n\mn3\mn k).2^k1^{n\mn3\mn k}$ for some $0\leq k<n- 3$ then 
$$\Psi_{[\blambda]_{B_6}}=[\blambda]_{B_6}+[(n\mn4\mn k).2^{k\pl1}1^{n\mn4\mn k}]_{B_6}.$$
\item If $\blambda=1^2.2^{n\mn4}$ then
$$\Psi_{[\blambda]_{B_6}}=[\blambda]_{B_6}+[1.2^{n\mn4}1]_{B_6}.$$
\item If $\blambda=k^2.2^{n\mn3\mn k}$ for some $2\leq k\leq n\mn3$ then
$$\Psi_{[\blambda]_{B_6}}=[\blambda]_{B_6}+[(k)(k\mn1).2^{n\mn3\mn k}1]_{B_6}+[(k\mn1)(k\mn2).2^{n\mn2\mn k}1]_{B_6}+[(k\mn2)^2.2^{n\mn1\mn k}]_{B_6}.$$
\item If $\blambda=(n\mn i\mn2)(n\mn j\mn1).2^{i\mn1}1^{j\mn i\mn1}$ for some $0<i<j\mn 1<n\mn 2$ then 
\begin{align*}
\Psi_{[\blambda]_{B_6}}&=[\blambda]_{B_6} + [(n\mn i\mn 3)(n\mn j\mn 1).2^i1^{j\mn i\mn2}]_{B_6} + [(n\mn i\mn 2)(n\mn j\mn 2).2^{i\mn 1}1^{j\mn i}]_{B_6}\\
&\quad  + [(n\mn i\mn 3)(n\mn j\mn 2).2^i1^{j\mn i\mn 1}]_{B_6}.
\end{align*}
\end{itemize}
\end{theorem}

\subsubsection{The principal series submatrix}\label{principal series dec} We will give explicit formulas in the theorem below, but first, for the sake of intuition we sketch a visual and conceptual way to describe the unipotent constituents of the projective indecomposable characters in the principal $\Phi_{2n}$-block that belong to the principal series,  inspired by \cite{ChTu}. One can make a simple graph with vertices the bipartitions $\blambda$ of $2n$ belonging to the principal $\Phi_{2n}$-block, and an edge between $\blambda$ and $\bmu$ if and only if $\bmu$ is obtained from $\blambda$ by moving either a single row or a single column of boxes preserving their charged contents mod $d$. In this case, we place $\bmu$ below $\blambda$ if $\bmu\prec_{\bs_0}\blambda$.
We refer the reader  to \cite[\S 5.4--5.5]{GrNo} for a picture of these posets (where they are drawn as simple directed graphs). The resulting graph is planar. 

\smallskip

One can then make a CW-complex where $2$-cells fill the diamond-and-triangle-shaped regions of the graph, each edge is a $1$-cell, and each vertex a $0$-cell. Each vertex $\blambda$ in the graph is then the top vertex of the closure of a cell of maximal possible dimension, or the top vertex of a union of closures of cells sharing that top vertex.  These shapes are either a diamond, a triangle, a single edge, a pair of edges (this happens only for $\blambda=1^{n+1}.1^{n-1}$), or simply the vertex itself. There turn out to be two such shapes that are points, and these are the two cuspidals. For the $\blambda$ that are top vertices of $1$- and $2$-dimensional shapes, taking all vertices in the boundary of the shape and reading off their labels $\bmu$ yields the unipotent constituents $[\bmu]$ of $\Psi_{[\blambda]}$. Moreover, those $\blambda$ at the top of $2$-dimensional shapes label the irreducible representations of the Hecke algebra, while the $\blambda$ that are only at the top of $1$-dimensional shapes label the simple modules with cuspidal support in some non-trivial, proper parabolic.  All of this is just an interpretation of the formulas in the following theorem, which are the same formulas as for the principal block of Category $\cO$ of the rational Cherednik algebra of $B_{2n}$ at parameters $(\frac{1}{2n},\frac{1}{2n})$ \cite{GrNo}.

\begin{theorem}\label{thm:principal series submatrix}
Let $\blambda$ be a bipartition of $2n$ such that $[\blambda]$ belongs to the principal $\Phi_{2n}$-block. If $\blambda=1^{2n}.$ or $.1^{2n}$ then $S_{[\blambda]}$ is cuspidal and so $\Psi_{[\blambda]}$ cannot be obtained by Harish-Chandra induction. Otherwise, the column of the decomposition matrix labeled by $\blambda$ is given by the appropriate choice of formula below:
\begin{itemize}
\item If $\blambda=(n + k)1^{n - k}.$ for some $1\leq k \leq n$, then 
$$
\Psi_{[(n\pl k)1^{n\mn k}.]}=[(n\pl k)1^{n\mn k}.]+[(n\pl k\mn 1)1^{n\mn k\pl 1}.]+[(n\mn 1)1^{n\mn k}.k\pl 1]+[(n\mn1)1^{n\mn k\pl 1}.k].
$$

\item If $\blambda=n1^n.$ then 
$$
\Psi_{[n1^n.]}=[n1^n.]+[(n\mn 1)1^n.1]+[(n\mn 1)1^{n\pl 1}.].
$$

\item If $\blambda=.(2n - k)1^k$ for some $0\leq k\leq n- 3$ then 
$$
\Psi_{[.(2n\mn k)1^k]}=[.(2n\mn k)1^k]+[n\mn 1\mn k.(n\pl 1)1^k]+[n\mn k\mn 2.(n\pl 1)1^{k\pl 1}]+[.(2n\mn k\mn 1)1^{k\pl 1}].
$$

\item If $\blambda=.(n+2)1^{n- 2}$ then 
$$
\Psi_{[.(n\pl 2)1^{n\mn 2}]}=[.(n\pl 2)1^{n\mn 2}]+[1.(n\pl 1)1^{n\mn 2}]+[.(n\pl 1)1^{n\mn 1}].
$$

\item If $\blambda=1^{k+1}.(n+1-k)1^{n-2}$ for some $0\leq k\leq n-1$ then 
$$
\Psi_{[1^{k\pl 1}.(n\pl 1\mn k)1^{n\mn2}]}=[1^{k\pl 1}.(n\pl 1\mn k)1^{n\mn2}]+[1^{k\pl 2}.(n\mn k)1^{n\mn2}]+[.(n\pl 1\mn k)1^{n\mn 1\pl k}] +[.(n\mn k)1^{n\pl k}].
$$ 

\item If $\blambda=(n-k)1^n.1^k$ for some $1\leq k\leq n-2$ then 
$$
\Psi_{[(n\mn k)1^n.1^k]}=[(n\mn k)1^n.1^k]+[(n\mn k\mn 1)1^n.1^{k\pl 1}]+[(n\mn k)1^{n\pl k}.]+[(n\mn k\mn 1)1^{n\pl k\pl 1}.].
$$

\item If $\blambda=(n-i)1^j.(n-j+1)1^{i-1}$ for some $1\leq i\leq n-2$ and some $0\leq j\leq n-1$ then 
$$
\Psi_{[\blambda]}=[(n\mn i)1^j.(n\mn j\pl 1)1^{i\mn 1}]+[(n\mn i)1^{j\pl 1}.(n\mn j)1^{i\mn 1}]+[(n\mn i\mn 1)1^j.(n\mn j\pl 1)1^i]+[(n\mn i\mn 1)1^{j\pl 1}.(n\mn j)1^i].
$$

\item If $\blambda=1^{n+1}.1^{n-1}$ then 
$$
\Psi_{[1^{n\pl 1}.1^{n\mn 1}]}=[1^{n\pl 1}.1^{n\mn 1}]+[1^{2n}.]+[.1^{2n}].
$$

\item If $\blambda=k1^{2n-k}.$ for some $2\leq k\leq n-1$ then 
$$
\Psi_{[(k)1^{2n\mn k}.]}=[k1^{2n\mn k}.]+[(k\mn 1)1^{2n\mn k\pl 1}.].
$$

\item If $\blambda=.k1^{2n- k}$ for some $2\leq k\leq n+1$ then 
$$
\Psi_{[.k1^{2n\mn k}]}=[.k1^{2n\mn k}]+[.(k\mn 1)1^{2n\mn k\pl 1}].
$$

\end{itemize}

\end{theorem}

\begin{proof} The first seven cases deal with $\blambda $ such that $\blambda$ labels an irreducible representation of the Hecke algebra. The formulas are then given by \cite{Fayers}. They may easily be checked to coincide with $F_i\Psi_{[\bmu]}$ for $\bmu$ and $i$ such that $\tilde{f}_i\bmu=\blambda$.

\smallskip

Let $L$ be the standard Levi subgroup of $G$ of type $A_{2n-1}$. 
If $\blambda=1^{n+1}.1^{n-1}$, we find that $\Psi_{[\blambda]}$ occurs as a summand of $R_{L}^{G}\Psi_{1^{2n}}$ by maximality of the family of $\blambda$ among the families of the three unipotent characters occurring in  $[\blambda]+[1^{2n}.]+[.1^{2n}]$, which equals  $R_{L}^{G}\Psi_{1^{2n}}$ cut to the principal $\Phi_{2n}$-block. Write 
$$\Psi_{[1^{n+1}.1^{n-1}]}=[1^{n+1}.1^{n-1}]+\alpha[1^{2n}.]+\beta [.1^{2n}]$$
for some $\alpha,\beta\in\{0,1\}$. We calculate that 
$$
e_n \Psi_{[1^{n+1}.1^{n-1}]}=[1^{n}.1^{n-1}]+\beta [.1^{2n-1}]
$$
which is only a projective character if $\beta\geq1$.
Then we calculate that
$$e_1\Psi_{[1^{n+1}.1^{n-1}]}=[1^{n+1}.1^{n-2}]+\alpha[1^{2n-1}.]$$
which is only a projective character if $\alpha\geq 1$. We conclude that $\alpha=\beta=1$.

\smallskip

Now we are ready to show that $S_{[1^{2n}.]}$ and $S_{[.1^{2n}]}$ are cuspidal. Since their projective covers are not summands of $R_{L}^{G}\Psi_{1^{2n}}$, these two simple modules are cuspidal if their Harish-Chandra restriction to a standard Levi subgroup of type $B\hskip-1mmC_{2n-1}$ is $0$. This is confirmed by checking that $\tilde{e}_i(1^{2n}.)=0=\tilde{e}_i(.1^{2n})$ for all $i\in\mathbb{Z}/2n \mathbb{Z}$. Indeed, for either bipartition, there is only one removable box, and it is canceled by the addable box of the same residue in the first row of the same component.

\smallskip

It remains to verify the formulas in the last two cases listed in the theorem. Consider first the case $\blambda=k1^{2n-k}.$ for some $2\leq k\leq n-1$. We find that $\tilde{e}_{k-1}\blambda=(k-1)1^{2n-k}.$, which belongs to a defect $1$ block. By  \cite{FoSr90}, $\Psi_{[(k-1)1^{2n-k}.]}=[(k-1)1^{2n-k}.]$.
 Therefore $\Psi_{[\blambda]}$ is a summand of $f_{k-1}\Psi_{[(k-1)1^{2n-k}.]}=f_{k-1}[(k-1)1^{2n-k}.]=[k1^{2n\mn k}.]+[(k\mn 1)1^{2n\mn k\pl 1}.]$. Now we argue by induction on $k$ that this is a projective indecomposable character. The base case is $k=2$: since $S_{[1^{2n}.]}$ is cuspidal, $\Psi_{[1^{2n}.]}$ is not a summand of the projective character $f_0\Psi_{[1^{2n-1}.]}=[21^{2n-2}.]+[1^{2n}.]$. The induction step says that $\Psi_{[k1^{2n\mn k}.]}$ has the desired formula, and therefore is not a summand of the projective character $f_{k}\Psi_{[k1^{2n\mn k\mn1}]}=
 [(k+1)1^{2n\mn k\mn 1}.]+[k1^{2n-k}.]$.
Finally, the argument for the case $\blambda=.k1^{2n-k}$ for some $2\leq k\leq n+1$ is similar.
\end{proof}

\subsection{Cuspidal columns}
In the previous section we have accounted for all the columns of the decomposition matrix of the principal $\Phi_{2n}$-block except the ones corresponding by unitriangularity to the unipotent characters 
$$\big[.1^{2n}\big],\quad \big[1^{2n}.\big],\quad \big[2^{n\mn1}.\big]_{B_2},\quad \big[.2^{n\mn3}\big]_{B_6}.$$
The corresponding projective indecomposable modules (PIMs) are the projective covers of cuspidal simple modules. The purpose of this section is to determine those remaining columns explicitly. As before, we will denote by 
$$\Psi_{[.1^{2n}]},\quad \Psi_{[1^{2n}.]},\quad \Psi_{[2^{n\mn1}.]_{B_2}},\quad \Psi_{[.2^{n\mn3}]_{B_6}}$$
the unipotent part of the characters of the corresponding PIMs. By unitriangularity of the decomposition matrix and maximality of $.1^{2n}$ in the partial order (by $a$-value, or by the reverse dominance order on families), it holds that $\Psi_{[.1^{2n}]}=\big[.1^{2n}\big]$, the Steinberg character. For the remaining three columns we determine the decomposition numbers using arguments from Deligne--Lusztig theory and from Kac--Moody categorification (truncated induction). The latter technique allows us to study the image of the corresponding PIMs under certain sequences of $i$-induction functors, providing us with an argument that all or all but one of the decomposition numbers (excepting the one on the diagonal) are zero. Then, two different Deligne--Lusztig characters give an upper and lower bound for the unique non-zero decomposition number, which turn out to coincide. Before treating each PIM in turn, we start with the following lemma.

\begin{lemma}\label{lem:zeroesfromiinduction}
Let $\Lambda$ be a symbol corresponding to a unipotent character $\big[\Lambda\big]$ of $G$. 
Let $\mathcal{S}$ be a set of symbols $\Lambda' < \Lambda$ such that
$$ \Psi_{[\Lambda]} \in [\Lambda] + \sum_{\Lambda' \in \mathcal{S}} \mathbb{N} [\Lambda'].$$
Let $\mathbf{i} = (i_1,i_2,\ldots,i_r)$ be a tuple of elements of $\mathbb{Z}/2n$ and let $\mathbf{i}^* = (i_r,i_{r-1},\ldots,i_1)$
be the reverse tuple. Assume that 
\begin{itemize}
 \item[(i)] $f_\mathbf{i}\Lambda  = \widetilde f_\mathbf{i}  \Lambda =  \Theta$ for some symbol $\Theta$;
 \item[(ii)] if $\Lambda' \in \mathcal{S}$ and $\Theta'$ occurs in $f_\mathbf{i}\Lambda'$ with $\Theta' < \Theta$ then $e_{\mathbf{i}^*} \Theta'= \Lambda'$.
\end{itemize}
Then $\Psi_{[\Lambda]}$ contains only $[\Lambda]$ and the characters $[\Lambda']$ for symbols $\Lambda'\in\mathcal{S}$ such that there exists $\Theta'$ occurring in $f_\mathbf{i}\Lambda'$ with $\Theta' < \Theta$. 
\end{lemma}

\begin{proof}
Recall that $P_{[\Lambda]}$ denotes the PIM corresponding to the unipotent character $[\Lambda]$ by unitriangularity, and that $\Psi_{[\Lambda]}$ is the unipotent part of its character. Since $i$-induction functors are exact, 
the module $f_\mathbf{i} P_{[\Lambda]}$ is projective, and by (i) it contains $P_{[\Theta]}$ as a direct summand.
Therefore any unipotent character $[\Theta'] \neq [\Theta]$ occurring in $\Psi_{[\Theta]}$ corresponds
to a symbol $\Theta'$ which satisfies the following 
two conditions:
\begin{itemize}
 \item there exists $\Lambda' \in \mathcal{S}$ such that $\Theta'$ occurs in $f_\mathbf{i} \Lambda'$;
 \item $\Theta' < \Theta$ (by unitriangularity of the decomposition matrix).
\end{itemize}
Now, by (ii), such a symbol satisfies $e_{\mathbf{i}^*} \Theta' = \Lambda'$. The result follows from 
the fact that $P_{[\Lambda]}$ is a direct summand of $e_{\mathbf{i}^*} P_{[\Theta]}$ by (i).
\end{proof}
 
\begin{remark}\label{rmk:fivskashiwara}
If $\Lambda$ has a unique addable $i$-box (i.e. $f_i \Lambda$ is a symbol) and no removable $i$-box (i.e. $e_i \Lambda = 0$) then $\widetilde f_{i}\Lambda  = f_{i}  \Lambda$ by \S\ref{sssec:branchingrules}. This is often helpful for checking condition (i) of Lemma~\ref{lem:zeroesfromiinduction}. 
\end{remark}

\subsubsection{The column corresponding to $\big[1^{2n}.\big]$} The symbol corresponding to this unipotent character is 
$$\begin{pmatrix}1&0&-1&\ldots&\widehat{\mn 2n\pl1}&\ldots\\ &&-1&\ldots&\ldots&\ldots\end{pmatrix}$$
so that the composition of the associated family is $(1,0,-1,-1,-2,-2,\ldots,\widehat{\mn 2n\pl1},\ldots)$. It dominates $(0,0,-1,-1,-2,-2,\ldots,\widehat{\mn 2n},\ldots)$, which is the family of the Steinberg character $\big[.1^{2n}]$,
and does not dominate any other family of unipotent characters of $G$. Therefore $\Psi_{[1^{2n}.]}=\big[1^{2n}.\big]+\alpha\big[.1^{2n}\big]$ for some $\alpha \geq 0$, by unitriangularity. We shall prove that $\alpha = 0$.

\begin{prop}\label{prop:PIM0}
We have 
$$\Psi_{[1^{2n}.]}=\big[1^{2n}.\big].$$
\end{prop}

\begin{proof}
We apply Lemma~\ref{lem:zeroesfromiinduction} to the symbol $\Lambda$ associated to $\big[1^{2n}.\big]$
and $\mathbf{i} = (0)$. We have
$$\Theta:= f_0 \Lambda  = f_0\begin{pmatrix}1&0&-1&\ldots&\widehat{\mn 2n\pl1}&\ldots\\ &&-1&\ldots&\ldots&\ldots\end{pmatrix} = \begin{pmatrix}1&0&-1&\ldots&\widehat{\mn 2n}&\ldots\\ &&-1&\ldots&\ldots&\ldots\end{pmatrix}.$$
By Remark~\ref{rmk:fivskashiwara} it is also equal to $\widetilde f_0 \Lambda$, therefore condition (i) of Lemma~\ref{lem:zeroesfromiinduction}
is satisfied.  On the other hand, the symbol $\Lambda'$ associated to $\big[.1^{2n}\big]$ is the unique symbol smaller
than $\Lambda$ and it satisfies
$$f_0\begin{pmatrix}0&-1&\ldots&\ldots&\ldots\\ 0&-1&\ldots&\widehat{\mn 2n}&\ldots \end{pmatrix}=\begin{pmatrix}1&&-1&\ldots&\ldots&\ldots\\ &0&-1&\ldots&\widehat{\mn 2n}&\ldots \end{pmatrix}.$$
Since $f_0 \Lambda$ and $f_0 \Lambda'$ lie in the same family, condition (ii) of Lemma~\ref{lem:zeroesfromiinduction}
is empty hence automatically satisfied. As a consequence $\big[.1^{2n}\big]$ is not a constituent of $\Psi_{[1^{2n}.]}$
and the proposition is proved.
\end{proof}

\subsubsection{The column corresponding to $\big[2^{n\mn1}.\big]_{B_2}$}
The unipotent character $\big[2^{n\mn1}.\big]_{B_2}$ of the $B_2$-series corresponds to the following
symbol of defect $-3$
$$\Lambda = \begin{pmatrix}&0&-1&\ldots&\widehat{\mn n\pl1}&\widehat{\mn n}&\ldots\\ 1&0&-1&\ldots&\ldots&\ldots&\ldots\end{pmatrix}.$$
The associated composition is $(1,0,0,-1,-1,\ldots,\mn n\pl1,\widehat{\mn n\pl1},\widehat{\mn n},\mn n,\ldots)$. From
the description in \S\ref{ssec:2nblock} one checks that the unipotent characters in the principal $\Phi_{2n}$-block 
lying in smaller families are $\big[1^{2n}.\big]$, $\big[.21^{2n\mn2}\big]$ and $\big[.1^{2n}\big]$.
Consequently there exist non-negative integers $\alpha$, $\beta$, and $\gamma$ such that
$$\Psi_{[2^{n\mn1}.]_{B_2}} = \big[2^{n\mn1}.\big]_{B_2} + \alpha \big[1^{2n}.\big] +  \beta \big[.21^{2n\mn2}\big]
+ \gamma \big[.1^{2n}\big].$$
The following theorem explicitly determines these integers whenever the $\ell$-part of $\Phi_{2n}(q)$
is not too small. 

\begin{theorem}\label{thm:PIM1}
There exists $\gamma \leq 2$ such that 
$$\Psi_{[2^{n\mn1}.]_{B_2}} = \big[2^{n\mn1}.\big]_{B_2} + \gamma \big[.1^{2n}\big].$$
Furthermore, if $\Phi_{2n}(q)_\ell > 4n$, then $\gamma = 2$. 
\end{theorem}

Note that by a result of Feit, the condition $\Phi_{2n}(q)_\ell>4n$ is satisfied for at least one prime number $\ell$ except for finitely many pairs $(q,n)$ \cite{Feit88}.

\begin{proof}
\textbf{Step 1}. 
The first step of the proof establishes that $\big[.1^{2n}\big]$ is the only unipotent character different from $\big[2^{n\mn1}.\big]$ that can occur in $\Psi_{[2^{n\mn1}.]_{B_2}}$ with a nonzero coefficient. For that purpose we use 
Lemma~\ref{lem:zeroesfromiinduction} with $\Lambda$ being the symbol attached to $[2^{n\mn1}.]_{B_2}$
and $\mathbf{i} = (1,2,\ldots,n-1)$. We have
$$ f_{\mathbf{i}} \Lambda = f_1 f_2 \cdots f_{n-1} \Lambda = \begin{pmatrix}&0&\ldots&\widehat{\mn n\pl1}&\ldots&\widehat{\mn 2n\pl1}&\ldots\\ 1&0&\ldots&\ldots&\ldots&\ldots&\ldots\end{pmatrix}$$
which also equals $\widetilde{f}_1 \widetilde{f}_2 \cdots \widetilde{f}_{n-1} \Lambda$ by 
Remark~\ref{rmk:fivskashiwara} so that condition (i)
of Lemma~\ref{lem:zeroesfromiinduction} holds. We list below the symbols obtained by inducing the 
ones associated to the unipotent characters $\big[1^{2n}.\big]$, $\big[.21^{2n\mn2}\big]$, and $\big[.1^{2n}\big]$.
$$ \begin{array}{c|c|c}
[\bmu] & \Lambda' & f_1 f_2 \cdots f_{n-1} \Lambda' \\[3pt]\hline & & \\[-5pt] 
\big[1^{2n}.\big] & \begin{pmatrix}1&0&\ldots&\widehat{\mn 2n\pl1}&\ldots\\ &&-1&\ldots&\ldots\end{pmatrix} & 
\begin{array}{l} 
\begin{pmatrix}1&0&\ldots&\ldots&\widehat{\mn 2n\pl1}&\ldots\\&0&\ldots&\widehat{\mn n\pl1}&\ldots&\ldots \end{pmatrix} \\
+  \begin{pmatrix}2&&0&\ldots&\ldots&\widehat{\mn 2n\pl1}&\ldots\\&&0&\ldots&\widehat{\mn n\pl2}&\ldots&\ldots \end{pmatrix}
\end{array}\\[30pt] \hline & & \\[-5pt]
\big[.21^{2n\mn2}\big] &\begin{pmatrix}&0&\ldots&\ldots&\ldots&\ldots\\ 1&&-1&\ldots&\widehat{\mn 2n\pl1}&\ldots\end{pmatrix}
  &  \begin{pmatrix}&0&\ldots&\ldots&\ldots&\ldots&\ldots\\ 1&0&\ldots&\widehat{\mn n\pl1}&\ldots&\widehat{\mn 2n\pl1}&\ldots\end{pmatrix} \\[12pt]  \hline & & \\[-5pt]
\big[.1^{2n}\big] & \begin{pmatrix}0&\ldots&\ldots&\ldots\\ 0&\ldots&\widehat{\mn 2n}&\ldots\end{pmatrix}
 &  \begin{pmatrix}0&\ldots&\ldots&\ldots\\0&\ldots&\widehat{-3n\pl1}&\ldots\end{pmatrix}
\end{array}
$$
The only symbol in the last column which is smaller than $f_\mathbf{i} \Lambda$ corresponds to the induction of 
the symbol $\Lambda'$ of $[.1^{2n}]$ (the last row in the table). One checks that $e_{\mathbf{i}^*} f_\mathbf{i} \Lambda' =
\Lambda'$ for that symbol, so that assumption (ii) of Lemma~\ref{lem:zeroesfromiinduction} is satisfied. We deduce that $[.1^{2n}]$
is indeed the only unipotent constituent of $\Psi_{[2^{n\mn1}.]_{B_2}}$ apart from $\big[2^{n\mn1}.\big]_{B_2}$.

\smallskip

\noindent \textbf{Step 2.} We now use the method in \cite{Du13} to show that $\gamma \leq 2$. This requires to
know how to decompose certain Deligne--Lusztig characters on the basis of PIMs, or at least to know the
coefficient of the PIMs corresponding to cuspidal modules. The following lemma will be useful to deal with
PIMs that can be obtained by induction from proper Levi subgroups of $G$. 

\begin{lemma}\label{lem:inducedsums}
Let $L$ (resp. $M$) be a $1$-split Levi of $G$ of type $B\hskip-1mmC_{2n-1}$ (resp. $A_{2n-1}$). We have 
$$\begin{aligned}
U_1 &\, := bR_L^G\Big(\sum_{i=1}^{2n-1} (-1)^{i-1} \big[i1^{2n\mn i\mn1}.\big] \Big) = \big[1^{2n}.\big] + \big[2n.\big] +(-1)^n \big[(n\mn1)1^{n}.1\big], \\
U_2&\, := bR_L^G\Big(\sum_{j=1}^{n-1} (-1)^{j+n} \big[(n\mn j)1^n.1^{j\mn1}\big] \Big) = (-1)^n \big[(n\mn1)1^{n}.1\big]  -\big[1^{n\pl1}.1^{n\mn1}\big],\\
U_3 &\, := bR_L^G\Big(\sum_{k=0}^{n-2} (-1)^{k+n} \big[2^k1^{n\mn k\mn1}.n\mn k\mn2\big]_{B_2} \Big) = (-1)^n\big[1^{n\mn1}.n\mn 1\big]_{B_2}  +\big[2^{n\mn1}.\big]_{B_2},\\
U_4&\, := bR_M^G\big(\big[1^{2n}\big]\big)= \big[1^{2n}.\big]+\big[.1^{2n}\big] +\big[1^{n\pl1}.1^{n\mn1}\big].\\
\end{aligned}$$
In particular the character $\big[2n.\big]-\big[.1^{2n}] = U_1-U_2-U_4$ is a combination of characters induced from proper Levi subgroups of $G$.
\end{lemma}  

\begin{proof}
The Harish-Chandra induction from $L$ to $G$ of a unipotent character associated to a bipartition $\blambda=\lambda^1.\lambda^2$ is described by adding one box to $\lambda^1$ or $\lambda^2$ in all possible ways so that the result is a bipartition. The formulas for $U_1$, $U_2$ and $U_3$ are easily deduced from that rule and the description of the principal $\Phi_{2n}$-block given in \S\ref{ssec:2nblock}. By \cite[Prop.~6.1.4]{GePf00} we have $R_M^G\big(\big[1^{2n}\big]\big) = \sum_{i}\big[1^i.1^{2n\mn i}\big]$ from
which we deduce the value of $U_4$ from \S\ref{ssec:2nblock}.
\end{proof}

With the notation in \S\ref{sssec:frg} we consider the Coxeter element $c = s_1 s_2 \cdots s_{2n}\in W_{2n}$. The decomposition of 
the Deligne--Lusztig character of $G$ attached to $c$ can be deduced from \cite[(3.2)]{FoSr86}. Most of its constituents do not belong to the principal $\Phi_{2n}$-block and we have 
$$bR_c = \big[2n.\big] + \big[.1^{2n}\big] +(-1)^{n-1} \big[1^{n\mn 1}.n\mn1\big]_{B_2}.$$
Using the combination of unipotent characters defined in Lemma~\ref{lem:inducedsums}, it can be rewritten as
$$\begin{aligned}
bR_c &\, = U_1-U_2-U_3-U_4+\big[2^{n-1}.\big]_{B_2} + 2\big[.1^{2n}] \\
 &\, = U_1-U_2-U_3-U_4 + \Psi_{[2^{n\mn1}.]_{B_2}} + (2-\gamma)\Psi_{[.1^{2n}]}.\\ 
\end{aligned}$$
Since the $U_i$'s are obtained by induction from proper Levi subgroups of $G$, they cannot
involve $\Psi_{[.1^{2n}]}$ since it corresponds to a cuspidal module. Therefore the coefficient
of $\Psi_{[.1^{2n}]}$ in $R_c$ is exactly $2-\gamma$. By \cite[Prop.~1.5]{Du13}, we must have $2-\gamma \geq 0$
therefore $\gamma \leq 2$.

\smallskip

\noindent \textbf{Step 3.}  In this final step of the proof, we show that $\gamma \geq 2$ under the assumption that $\Phi_{2n}(q)_\ell > 4n$. We consider the element $w:=c^2 = (s_1s_2\cdots s_{2n})^2$. By \cite[\S5.2]{Sp74}, it is a $2n$-regular element. Therefore any torus $\mathbf{T}_w$ of type $w$ is the centraliser of a $\Phi_{2n}$-Sylow subgroup of $\bfG$ and the constituents of the Deligne--Lusztig character $R_w$ are exactly the unipotent characters in the principal $\Phi_{2n}$-block, see \cite[Thm. 5.24]{BMM93}. The exact decomposition of $R_w$ can be deduced for example from \cite[(3.2)]{FoSr86}. We find in particular 
$$\big\langle R_w;\big[2^{n\mn1}.\big]_{B_2}\big\rangle_G=-2 \quad \text{and} \quad \big\langle R_w;\big[.1^{2n}\big]\big\rangle_G=1.$$
On the other hand, if $\Phi_{2n}(q)_\ell > 4n$ then there exists by \cite[Ex.~1.17]{DuMa20} an $\ell$-character of $T_w$ in general position. By \cite[Lem.~1.13]{DuMa20} this forces $ \big\langle R_w; \Psi_{[2^{n\mn1}.]_{B_2}} \rangle_G \geq 0$, which gives
$$ 0 \leq \big\langle R_w; \Psi_{[2^{n\mn1}.]_{B_2}} \rangle_G = \big\langle R_w; \big[2^{n\mn1}.\big]_{B_2} \rangle_G + \gamma \big\langle R_w;\big[.1^{2n}\big]\big\rangle_G = -2 + \gamma$$
and therefore proves that $\gamma \geq 2$. Note that it also shows that $\Psi_{[.1^{2n}]}$ does not occur in $R_c$.
\end{proof}

\subsubsection{The column corresponding to $\big[.2^{n\mn3}\big]_{B_6}$}
We now focus on the last column, corresponding to the unipotent character $\big[.2^{n\mn3}\big]_{B_6}$ of the $B_6$-series. It corresponds to the following symbol of defect $5$
$$\Lambda = \begin{pmatrix}2&1&0&-1&\ldots&\ldots&\ldots&\ldots\\&&&-1&\ldots&\widehat{\mn n\pl2}&\widehat{\mn n\pl1}&\ldots\end{pmatrix}
$$
whose associated composition is $(2,1,0,-1,-1,\ldots,\widehat{\mn n\pl2},\widehat{\mn n\pl1},\ldots)$. From
the description in \S\ref{ssec:2nblock} one checks that it is minimal among the unipotent characters in the $B_6$-series lying in the principal $\Phi_{2n}$-block. There are 3 unipotent characters in the principal $\Phi_{2n}$-block belonging to the $B_2$-series 
which are smaller for the order on families. We list these characters with their associated symbol:
\begin{align*}
\big[2^{n\mn1}.\big]_{B_2} &\, \longleftrightarrow \begin{pmatrix}&0&\ldots&\widehat{\mn n\pl1}&\widehat{\mn n}& \ldots\\ 1&0&\ldots&\ldots&\ldots & \ldots\end{pmatrix},\\
\big[2^{n\mn2}1.1\big]_{B_2}& \, \longleftrightarrow\begin{pmatrix}&&0&\ldots&\widehat{\mn n\pl2}&\mn n\pl1&\widehat{\mn n}& \ldots\\2&&0&\ldots&\ldots&\ldots&\ldots& \ldots\end{pmatrix},\\
\big[2^{n\mn2}.1^2\big]_{B_2}&\, \longleftrightarrow \begin{pmatrix}&&0&\ldots&\widehat{\mn n\pl2}&\widehat{\mn n\pl1}& \ldots\\ 2&1&&-1&\ldots &\ldots& \ldots
\end{pmatrix}.
\end{align*}
Finally, there are 8 unipotent characters in the principal series belonging to the principal $\Phi_{2n}$-block which are smaller than 
$\big[.2^{n\mn3}\big]_{B_6}$:
\begin{align*}
\big[21^{2n\mn2}.\big] &\, \longleftrightarrow 
\begin{pmatrix}2&0&\ldots&\widehat{\mn 2n\pl2}\\&&-1&\ldots
\end{pmatrix},
&\big[.31^{2n\mn3}\big]&\, \longleftrightarrow 
\begin{pmatrix}&0&\ldots&\ldots&\ldots\\2&&-1&\ldots&\widehat{\mn 2n\pl2}
\end{pmatrix},\\
\big[.21^{2n\mn2}\big]&\, \longleftrightarrow \begin{pmatrix}
&0&\ldots&\ldots&\ldots\\
1&&-1&\ldots&\widehat{\mn 2n\pl1}
\end{pmatrix},
&\big[.1^{2n}\big]&\, \longleftrightarrow \begin{pmatrix}
0&\ldots&\ldots\\
0&\ldots&\widehat{\mn 2n}
\end{pmatrix},\\
 \big[1^{2n}.\big]&\, \longleftrightarrow \begin{pmatrix}
1&\ldots&\widehat{\mn 2n\pl1}\\
&-1&\ldots
\end{pmatrix},
&\big[21^n.1^{n\mn2}\big]&\, \longleftrightarrow \begin{pmatrix}
2&0&\ldots&\ldots&\widehat{\mn n}\\
&0&\ldots&\widehat{\mn n\pl2}&\ldots
\end{pmatrix},\\
\big[1^n.21^{n\mn2}\big]&\, \longleftrightarrow \begin{pmatrix}
1&\ldots&\ldots&\widehat{\mn n\pl1}\\
1&-1&\ldots&\widehat{\mn n\pl1}
\end{pmatrix},
& \big[1^{n\pl1}.1^{n\mn1}\big]&\, \longleftrightarrow \begin{pmatrix}
1&\ldots&\ldots&\ldots&\widehat{\mn n}\\
&0&\ldots&\widehat{\mn n\pl1}&\ldots
\end{pmatrix}.
\end{align*}
We will show that none of these characters contribute to $\Psi_{[.2^{n\mn3}]_{B_6}}$ with the exception
of $\big[1^{2n}.\big]$

\begin{theorem}\label{thm:PIM2}
There exists $\beta \geq 0$ such that
$$\Psi_{[.2^{n\mn3}]_{B_6}} =\big[.2^{n\mn3}\big]_{B_6}+\beta\big[1^{2n}.\big].$$
Furthermore, if $\Phi_{2n}(q)_\ell > 4n$ then $\beta =2$.
\end{theorem}

\begin{proof}
The proof follows that of Theorem~\ref{thm:PIM1}, but more computations are needed since more characters are involved. 

\smallskip

\noindent {\bf Step 1.} We start by computing the image of the various symbols involved under the operator 
$f_\mathbf{i}$ for $\mathbf{i} = (1,0)$. For the symbol $\Lambda$ corresponding to the unipotent character in the $B_6$-series we have
\begin{equation}
\label{eq:indB6}
f_\mathbf{i} \Lambda = \begin{pmatrix}2&1&0&-1&\ldots&\ldots&\ldots&\ldots\\&&&-1&\ldots&\widehat{\mn n\pl1}&\widehat{\mn n}&\ldots\end{pmatrix}
\end{equation}
which also equals  $\widetilde{f}_\mathbf{i} \Lambda$ by Remark~\ref{rmk:fivskashiwara}. For the symbols
corresponding to characters in the $B_2$-series we obtain
$$ \begin{array}{c|c}
[\Lambda'] & {f}_\mathbf{i} \Lambda' \\\hline
\big[2^{n\mn1}.\big]_{B_2} & \begin{pmatrix}2 & & &-1 & \ldots&\widehat{\mn n\pl1}&\widehat{\mn n}& \ldots\\ & 1&0&-1&\ldots&\ldots&\ldots & \ldots\end{pmatrix}\\
\big[2^{n\mn2}1.1\big]_{B_2}& \begin{pmatrix}2 &&&-1 & \ldots&\widehat{\mn n\pl2}&\mn n\pl1&\widehat{\mn n}& \ldots\\ 2&&0&-1 & \ldots&\ldots&\ldots&\ldots& \ldots\end{pmatrix}\\
\big[2^{n\mn2}.1^2\big]_{B_2}& \begin{pmatrix}2 && &-1 & \ldots&\widehat{\mn n\pl2}&\widehat{\mn n\pl1}& \ldots\\ 2&1&&-1 & \ldots&\ldots &\ldots& \ldots
\end{pmatrix}
\end{array}$$
None of these symbol is smaller than $f_\mathbf{i} \Lambda$. The case of the principal series characters
is given in the following table:
$$\begin{array}{c|c||c|c}
[\Lambda'] & {f}_\mathbf{i} \Lambda' & [\Lambda'] & {f}_\mathbf{i} \Lambda'\\\hline
\big[21^{2n\mn2}.\big] & 
\begin{pmatrix}2&1& -1 & \ldots&\widehat{\mn 2n\pl1}\\&&-1&\ldots & \ldots
\end{pmatrix}
&\big[.31^{2n\mn3}\big]& 
\begin{pmatrix}2& &-1&\ldots&\ldots\\2&&-1&\ldots&\widehat{\mn 2n\pl2}
\end{pmatrix}\\
\big[.21^{2n\mn2}\big]& \begin{pmatrix}
2 & & -1&\ldots&\ldots\\
& 1&-1&\ldots&\widehat{\mn 2n\pl1}
\end{pmatrix}
&\big[.1^{2n}\big]& \begin{pmatrix}
2 & &-1 & \ldots&\ldots\\
& 0&\ldots & \ldots&\widehat{\mn 2n}
\end{pmatrix}\\
 \big[1^{2n}.\big]& \begin{pmatrix}
2 & 0 &\ldots&\widehat{\mn 2n} &\ldots\\
&& -1&\ldots &\ldots
\end{pmatrix}
&\big[21^n.1^{n\mn2}\big]& \begin{pmatrix}
2&1 & & -1 & \ldots&\ldots&\widehat{\mn n}\\
& &0& \ldots & \ldots&\widehat{\mn n\pl1}&\ldots
\end{pmatrix}\\
\big[1^n.21^{n\mn2}\big]& \begin{pmatrix}
2 & &0&\ldots&\ldots&\widehat{\mn n\pl1}\\
& 1&& -1&\ldots&\widehat{\mn n}
\end{pmatrix}
& \big[1^{n\pl1}.1^{n\mn1}\big]& \begin{pmatrix}
2&\ldots&\ldots&\ldots&\widehat{\mn n}\\
&0&\ldots&\widehat{\mn n}&\ldots
\end{pmatrix}
\end{array}$$
The induction of the symbols corresponding to the characters $\big[1^n.21^{n\mn2}\big]$, $\big[.31^{2n\mn3}\big]$,
$ \big[21^n.1^{n\mn2}\big]$ and $\big[1^{n\pl1}.1^{n\mn1}\big]$ are not strictly dominated by  the symbol
$f_\mathbf{i} \Lambda$ computed in \eqref{eq:indB6}. The other symbols $\Lambda'$ satisfy $e_{\mathbf{i}^*} f_\mathbf{i}\Lambda' = \Lambda'$, therefore Lemma~\ref{lem:zeroesfromiinduction} shows that there exist non-negative integers
$\beta_1$, $\beta_2$, $\beta_3$ and $\beta_4$ such that
$$\Psi_{[.2^{n\mn3}]_{B_6}} =\big[.2^{n\mn3}\big]_{B_6}+\beta_1\big[21^{2n\mn2}.\big] + \beta_2\big[.21^{2n\mn2}\big] + \beta_3\big[1^{2n}.\big] + \beta_4\big[.1^{2n}\big].$$
We now consider the induction with respect to the sequence $\mathbf{i} = (-1,-2,0,-1,1,0)$. We find
\begin{equation}
\label{eq:indB6bis}
\widetilde{f}_\mathbf{i}\Lambda = f_\mathbf{i} \Lambda = \begin{pmatrix}2&1&0&-1&\ldots&\ldots&\ldots&\ldots\\&&&-1&\ldots&\widehat{\mn n\mn1}&\widehat{\mn n\mn2}&\ldots\end{pmatrix}
\end{equation}
using Remark~\ref{rmk:fivskashiwara}. For the four remaining principal series characters we obtain the following:
$$\begin{array}{c|c|c}
[\Lambda'] & \Lambda' & f_{\mathbf{i}} \Lambda'  \\\hline
\big[21^{2n\mn2}.\big] & \begin{pmatrix}2&0&\ldots&\widehat{\mn 2n\pl2}\\&&-1&\ldots \end{pmatrix}& 
 \begin{pmatrix}2&1 & 0&-1 & -3 & \ldots&\widehat{\mn 2n\mn1}\\& & &-1&\ldots & \ldots & \ldots \end{pmatrix}
\\
\big[.21^{2n\mn2}\big] & \begin{pmatrix} &0&\ldots&\ldots&\ldots\\ 1&&-1&\ldots&\widehat{\mn 2n\pl1} \end{pmatrix} &  \begin{pmatrix} 2&1 & 0 & & -3 & \ldots\\ & 1&&  -1&\ldots&\widehat{\mn 2n\pl1} \end{pmatrix} \\
\big[1^{2n}.\big] &\begin{pmatrix} 1&\ldots&\widehat{\mn 2n\pl1}\\ &-1&\ldots \end{pmatrix}& 
\begin{pmatrix} 2& 1 & 0 &&  -2&  \ldots&\widehat{\mn 2n\mn2}\\ & & & -1&\ldots &\ldots&\ldots\end{pmatrix}\\
\big[.1^{2n}\big] & \begin{pmatrix} 0&\ldots&\ldots\\ 0&\ldots&\widehat{\mn 2n}  \end{pmatrix}&  \begin{pmatrix} 2 & 1 & 0 &  -3 & \ldots&\ldots\\ & & 0&\ldots & \ldots&\widehat{\mn 2n}  \end{pmatrix} \\
\end{array}$$
Note that for the computations we can use that $f_i$ and $f_j$ commute whenever $i \ncong j\pm1$ modulo $2n$ so 
that $f_\mathbf{i} = (f_{-1} f_0 f_1) (f_{-2} f_{-1} f_0)$. The induction of the 
symbols corresponding to the unipotent characters $\big[.21^{2n\mn2}\big] $ and $\big[.1^{2n}\big]$ 
are not strictly dominated by the symbol $f_\mathbf{i} \Lambda$ computed in \eqref{eq:indB6bis}. The other symbols satisfy
the conditions of Lemma~\ref{lem:zeroesfromiinduction} which shows that $\beta_2 = \beta_4 = 0$. It remains to
show that $\beta_1 = 0$. For that purpose we use again Lemma~\ref{lem:zeroesfromiinduction} for $\Lambda$
and $\mathbf{i} = (-n+2,\ldots,-2,-1,0)$. The induced symbols are given in the following table.
$$\begin{array}{c|c|c}
[\Lambda'] & \Lambda' & f_{\mathbf{i}} \Lambda'  \\\hline
\big[.2^{n\mn3}\big]_{B_6} &  \begin{pmatrix}2&\ldots&\ldots&\ldots &\ldots \\&-1&\ldots&\widehat{\mn n\pl2}&\widehat{\mn n\pl1}\end{pmatrix} & \begin{pmatrix}2& \ldots&\ldots&\ldots & \ldots & \ldots\\&-1&\ldots&\widehat{\mn n\pl2}& \ldots & \widehat{\mn 2n\pl2}\end{pmatrix}  \\
\big[21^{2n\mn2}.\big] & \begin{pmatrix}2&0&\ldots&\widehat{\mn 2n\pl2}\\&&-1&\ldots \end{pmatrix}& 
 \begin{pmatrix}2& \ldots& \widehat{\mn n\pl2} & \ldots &\widehat{\mn 2n\pl2}\\& -1&\ldots & \ldots & \ldots \end{pmatrix}
\\
\big[1^{2n}.\big] &\begin{pmatrix} 1&\ldots&\widehat{\mn 2n\pl1}\\ &-1&\ldots \end{pmatrix}& 
\begin{pmatrix} 1&\ldots&\widehat{\mn 3n\pl2}\\ &-1&\ldots \end{pmatrix}\\
\end{array}$$
The induced symbols corresponding to the characters $\big[.2^{n\mn3}\big]_{B_6}$ and $\big[21^{2n\mn2}.\big]$
lie in the same family. In addition, if $\Lambda'$ is the symbol attached to $\big[1^{2n}.\big]$ then $e_{\mathbf{i}^*}f_\mathbf{i} \Lambda' = \Lambda'$. Finally, one can check that $\widetilde{f}_\mathbf{i} \Lambda = f_\mathbf{i} \Lambda$
using \S\ref{sssec:branchingrules}  and invoke Lemma~\ref{lem:zeroesfromiinduction} to conclude that $\beta_1 = 0$ (note that Remark~\ref{rmk:fivskashiwara} does not apply to the last step of the computation of $\widetilde{f}_\mathbf{i} \Lambda$). For the remainder of the proof we will write $\beta := \beta_3$. 
\smallskip

\noindent {\bf Step 2.} Assume now that $\Phi_{2n}(q)_\ell > 4n$. Then 
$\Psi_{[2^{n\mn1}.]_{B_2}}=\big[2^{n-1}.\big]_{B_2} + 2\big[.1^{2n}]$ by Theorem~\ref{thm:PIM1}.
In addition, it was observed in the proof of that theorem that the Deligne--Lusztig character associated to a Coxeter element $c$
decomposes as
$$R_c = U_1-U_2-U_3-U_4 + \Psi_{[2^{n\mn1}.]_{B_2}}.$$
This shows that apart from $ \Psi_{[2^{n\mn1}.]_{B_2}}$, none of the PIMs corresponding to
cuspidal modules appear in $R_c$. In order to decompose the next Deligne--Lusztig 
characters we shall use the following identities, whose proofs are identical 
to those for Lemma~\ref{lem:inducedsums}.

\begin{lemma}\label{lem:inducedsums2}
Let $L$ be a $1$-split Levi of $G$ of type $B_{2n-1}$. We have 
$$\begin{aligned}
U_5 &\, := bR_L^G\Big(\big[ 2n\mn1. \big]\Big) = \big[2n.\big] + \big[(2n\mn1)1.\big], \\
U_6 & \, := bR_L^G\Big( \big[1^{n\mn1}.n\mn2\big]_{B_2}+\big[1^{n\mn2}.n\mn1\big]_{B_2}\Big) = 
2\big[1^{n\mn1}.n\mn1\big]_{B_2} + \big[21^{n\mn2}.n\mn2\big]_{B_2} + \big[1^{n\mn2}.(n\mn1)1\big]_{B_2}, \\
U_7 & \, := bR_L^G\Big(\big[(2n\mn2)1. \big]\Big) = \big[(2n\mn1)1.\big] + \big[(2n\mn2)1^2.\big], \\
U_8 & \, := bR_L^G\Big(\sum_{j=1}^{n} (-1)^{j+n} \big[1^j.(n\pl1\mn j)1^{n\mn2}\big] \Big) = (-1)^{n+1} \big[1.(n\pl1)1^{n-2}.\big]  + \big[1^{n\pl1}.1^{n\mn1}\big],\\
U_9 & \, := bR_L^G\Big(\big[1^{n\mn3}.(n\mn1)1\big]_{B_2}+\big[1^{n\mn2}.(n\mn2)1\big]_{B_2}+\big[21^{n\mn2}.n\mn3\big]_{B_2}-2 \big[1^{n\mn2}.n\mn1\big]_{B_2}- \big[1^{n\mn1}.n\mn2\big]_{B_2}\Big), \\
 &\, \phantom{:} = \big[1^{n\mn3}.(n\mn1)2\big]_{B_2} + \big[21^{n\mn3}.(n\mn2)1\big]_{B_2} + \big[2^21^{n\mn3}.n\mn3\big]_{B_2} - 3\big[1^{n\mn1}.n\mn1\big]_{B_2}, \\
U_{10} &\, := bR_L^G\Big(\sum_{r=0}^{n-4} (-1)^{r+n} \big[n\mn r\mn4.2^r1^{n\mn r\mn3}\big]_{B_6} \Big) = (-1)^n\big[n\mn 3.1^{n\mn3}\big]_{B_6}  +\big[.2^{n\mn3}\big]_{B_6}.\\
\end{aligned}$$ 
\end{lemma}  

Since we want to use the result in \cite{Du13} to get information on the PIMs with cuspidal head,
it is enough to consider the Deligne--Lusztig characters $R_w$ for elements $w\in W_{2n}$ whose conjugacy 
class does not meet any proper parabolic subgroup. Such classes are called
cuspidal (or elliptic) and are described in \cite[Prop.~3.4.6]{GePf00}. With the notation in 
\cite[\S3.4.2]{GePf00}, minimal length elements of the first 3 cuspidal classes, ordered by length, are given
by $c =w_{(2n)}^-$, $v_n = s_1 s_2 c = w_{(1,2n-1)}^-$, and
$w_n := s_1 s_2 s_3 s_2 c = w_{(2,2n-2)}$. Let us first consider the Deligne--Lusztig character associated
to $v_n$.  
From the decomposition of $R_{v_n}$ given in Lemma~\ref{lem:vn} and the characters defined in Lemma~\ref{lem:inducedsums} and \ref{lem:inducedsums2}
we have
$$\begin{aligned}
bR_{v_n} &\, = 2(U_1-U_2-U_3-U_4)-U_5-D_G(U_5)+(-1)^nU_6 +2\big[2^{n-1}.\big]_{B_2} + 4\big[.1^{2n}] \\
 &\, = 2(U_1-U_2-U_3-U_4)-U_5-D_G(U_5)+(-1)^nU_6+ 2\Psi_{[2^{n\mn1}.]_{B_2}}.\\ 
\end{aligned}$$
Again, apart from $ \Psi_{[2^{n\mn1}.]_{B_2}}$, none of the PIMs corresponding to
cuspidal modules appear in $R_{v_n}$. We move on to the element $w_n := s_1 s_2 s_3 s_2 c$.
Using Lemma~\ref{lem:wn} and Lemma~\ref{lem:inducedsums} we can decompose $R_{w_n}$ as
$$\begin{aligned}
bR_{w_n}  =&\, 3U_1-3U_2 -U_3-4U_4-2U_5+U_7 +U_8+D_G(U_1)-2D_G(U_5)+D_G(U_7)\\
 &\, + (-1)^{n-1}U_9-U_{10}+\big[.2^{n-3}\big]_{B_6}+3\big[2^{n-1}.\big]_{B_2}+2\big[1^{2n}.\big]+3\big[.1^{2n}\big]\\[4pt]
  = &\,  3U_1-3U_2 -U_3-4U_4-2U_5+U_7 +U_8+D_G(U_1)-2D_G(U_5)+D_G(U_7)\\
 &\, + (-1)^{n-1}U_9-U_{10}+\Psi_{[.2^{n-3}]_{B_6}}+3\Psi_{[2^{n-1}.]_{B_2}}+(2-\beta)\Psi_{[1^{2n}.]}.
\end{aligned}$$
Here we have also used Proposition~\ref{prop:PIM0} and Theorem~\ref{thm:PIM1} which give the decomposition of $\Psi_{[2^{n-1}.]_{B_2}}$ and $\Psi_{[1^{2n}.]}$. It follows from \cite[Prop.~1.5]{Du13} that $2-\beta \geq 0$ hence $\beta \leq 2$.
\smallskip

\noindent {\bf Step 3.} The argument is entirely similar to that given in the proof of Theorem~\ref{thm:PIM1} with the exception that one uses the multiplicities
$$\big\langle R_w;\big[.2^{n-3}\big]_{B_6}\big\rangle_G=-2 \quad \text{and} \quad \big\langle R_w;\big[1^{2n}.\big]\big\rangle_G=1.$$
This shows that $\beta \geq 2$, hence $\beta =2$.
\end{proof}

\appendix
\renewcommand{\thesection}{A}

\section*{Appendix: Computation of Deligne--Lusztig characters}

We fix an integer $n$. As in \S\ref{sssec:frg} we denote by $W_{2n}$ the Weyl group of type $B_{2n}$ with generators $s_1,\ldots,s_{2n}$. We explain here how to compute the Deligne--Lusztig characters associated with the elements
$$\begin{aligned}
&& v_n &\, :=  (s_1 s_2) (s_1 s_2s_3 \cdots s_{2n-1}s_{2n}) &&\\
\text{and} \qquad\qquad && w_n &\, := (s_1 s_2 s_3 s_2) (s_1 s_2 s_3 \cdots s_{2n-1}s_{2n})&&  \qquad \qquad
\end{aligned}$$
of respective lengths $2n+2$ and $2n+4$. 

\subsection{The Deligne--Lusztig character associated to $v_n$}
Using the notation in \cite[\S3.4.2]{GePf00}, the element decomposes as $v_n = w_{(1,2n-1)}^- = s_1 b_{1,2n-1}^- = b_{0,1}^- b_{1,2n-1}^-$ and one can compute the values of irreducible characters of $W_{2n}$ at the element $v_n$ using the Murnaghan--Nakayama rule.
More precisely, if $\blambda$ is a bipartition of $2n$ and $\chi_{\blambda}$ is the corresponding irreducible character of $W_{2n}$ then by \cite[Thm.~10.3.1]{GePf00} we have
\begin{equation}\label{eq:MNrule}
 \chi_{\blambda}(v_n) = \sum_{\gamma} (-1)^{h(\gamma)} \chi_{\blambda \smallsetminus \gamma}(s_1)
 \end{equation}
where $\gamma$ runs over all the $(2n\mn1)$-hooks of $\blambda$, and $h(\gamma)$ equals the leg length of $\gamma$ plus $1$ if $\gamma$ is in the second component. From the description of the principal $\Phi_{2n}$-block
$B$ given in \S\ref{ssec:2nblock} one can determine the bipartitions $\blambda$ which have a $(2n\mn1)$-hook and whose
family has a non-trivial intersection with $\mathrm{Irr}\, B$. They are listed in Table~\ref{tab:bipartvn}, together with
the value of $h(\gamma)$, the $(2n\mn1)$-core $\blambda \smallsetminus \gamma$, the value of $\chi_{\blambda \smallsetminus \gamma}(s_1s_2)$ and that of $\chi_{\blambda}(w_n)$ using \eqref{eq:MNrule}. Note that the second
half of the table can be obtained from the first using the relation $\chi_{\blambda^*}(v_n) = (-1)^{l(v_n)} \chi_{\blambda}(v_n) = \chi_{\blambda}(v_n)$, where $l(w)$ is the Coxeter length of $w\in W_{2n}$.

\begin{table}
$$ \begin{array}{c|c|c|c|c}
\blambda & h(\gamma) & \blambda \smallsetminus \gamma & \chi_{\blambda \smallsetminus \gamma}(s_1) & \chi_{\blambda}(v_n) \\\hline
2n. & 0 & \multirow{5}{*}{$1. $} & \multirow{5}{*}{$1$}& 1 \\
n21^{n\mn2}. & n-1 & & & (-1)^{n-1}  \\
1^{2n}.   & 2n-2  & & & 1 \\ 
1.(n\pl1)1^{n\mn2} & n-2+1& & & (-1)^{n-1} \\ 
1.1^{2n\mn1} & 2n-2+1 & & & -1\\\hline
.1^{2n}  &2n-2+1 & \multirow{5}{*}{$.1 $} & \multirow{5}{*}{$-1$} & 1\\
.n21^{n\mn2} & n-1+1& & &  (-1)^{n-1}  \\
.2n & 0+1 & & & 1   \\ 
(n\mn1)1^n.1& n   & & & (-1)^{n-1}  \\
2n\mn1.1 & 0 & & & -1
 \end{array}$$
\caption{Bipartitions $\blambda$ with a $(2n\mn1)$-hook whose family intersects the principal $\Phi_{2n}$-block}
\label{tab:bipartvn}
\end{table}

\smallskip

The families corresponding to the bipartitions $2n.$ and $.1^{2n}$ contain only one unipotent character each,
the trivial and the Steinberg character respectively. The other families have 4 elements and are listed in Table~\ref{tab:families4vn}.
By convention the special character is the first one in each list. The corresponding almost characters can be computed
from \cite[\S4]{Lu84}. They are given in Table~\ref{tab:almostvn}.

\begin{table}
$$ \begin{array}{c|c}
\blambda & \mathfrak{F} \\\hline
\vphantom{\Big)} 
n21^{n\mn2}. &  \big[n.21^{n\mn2}\big], \big[n21^{n\mn2}. \big], \big[1.(n\pl1)1^{n\mn2}\big], \big[1^{n\mn2}.(n\mn1)1\big]_{B_2} \\[5pt]
1^{2n}.   &   \big[1 . 1^{2n\mn1}\big] ,\big[1^{2n} .\big], \big[. 21^{2n\mn2}\big], [.1^{2n\mn2}\big]_{B_2}  \\[5pt]
1.(n\pl1)1^{n\mn2} &    \big[n.21^{n\mn2}\big], \big[n21^{n\mn2}. \big], \big[1.(n\pl1)1^{n\mn2}\big], \big[1^{n\mn2}.(n\mn1)1\big]_{B_2}  \\[5pt]
1.1^{2n\mn1} & \big[1 . 1^{2n\mn1}\big] ,\big[1^{2n} .\big], \big[. 21^{2n\mn2}\big], [.1^{2n\mn2}\big]_{B_2}  \\[5pt]
.n21^{n\mn2} &  \big[(n\mn1)1.1^n\big], \big[.n21^{n\mn2}\big], \big[(n\mn1)1^n.1\big], \big[21^{n\mn2}.n\mn2\big]_{B_2} \\[5pt]
.2n & \big[2n\mn1. 1\big], \big[ . 2n\big], [(2n\mn1)1. \big], \big[2n\mn2. \big]_{B_2}     \\[5pt] 
(n\mn1)1^n.1&  \big[(n\mn1)1.1^n\big], \big[.n21^{n\mn2}\big], \big[(n\mn1)1^n.1\big], \big[21^{n\mn2}.n\mn2\big]_{B_2} \\[5pt]    
2n\mn1.1 &  \big[2n\mn1. 1\big], \big[ . 2n\big], [(2n\mn1)1. \big], \big[2n\mn2. \big]_{B_2} \\
 \end{array}$$
\caption{Families $\mathfrak{F}$  with $4$ elements occurring in $bR_{v_n}$}
\label{tab:families4vn}
\end{table}

\begin{table}
$$ \begin{array}{c|c}
\blambda &bR_{\chi_{\blambda}} \\\hline
\vphantom{\Big)} 
n21^{n\mn2}. &   \displaystyle  - \frac{1}{2} \Big( \big[1.(n\pl1)1^{n\mn2}\big]+ \big[1^{n\mn2}.(n\mn1)1\big]_{B_2}\Big) \\[5pt]
1^{2n}.   &    \displaystyle  \frac{1}{2} \Big(\big[1^{2n} .\big]- \big[. 21^{2n\mn2}\big]\Big)  \\[5pt]
1.(n\pl1)1^{n\mn2} &    \displaystyle  \frac{1}{2} \Big( \big[1.(n\pl1)1^{n\mn2}\big]- \big[1^{n\mn2}.(n\mn1)1\big]_{B_2}\Big) \\[5pt]
1.1^{2n\mn1} & \displaystyle  \frac{1}{2} \Big(\big[1^{2n} .\big]+ \big[. 21^{2n\mn2}\big]\Big)  \\[5pt]
.n21^{n\mn2} &   \displaystyle  -  \frac{1}{2} \Big(\big[(n\mn1)1^n.1\big]+\big[21^{n\mn2}.n\mn2\big]_{B_2}\Big) \\[5pt]
.2n &  \displaystyle  \frac{1}{2} \Big( \big[ . 2n\big]- \big[(2n\mn1)1. \big]\Big)     \\[5pt] 
(n\mn1)1^n.1&   \displaystyle  \frac{1}{2} \Big(\big[(n\mn1)1^n.1\big]-\big[21^{n\mn2}.n\mn2\big]_{B_2}\Big) \\[5pt]
2n\mn1.1 & \displaystyle  \frac{1}{2} \Big( \big[(2n\mn1)1. \big]- \big[ . 2n\big]\Big)  \\
 \end{array}$$
\caption{Some almost characters occurring in $bR_{v_n}$}
\label{tab:almostvn}
\end{table}

Using the data in these tables, together with the formula $R_{w} = \sum_{\blambda} \chi_{\blambda}(w) R_{\chi_{\blambda}} $,
we obtain the explicit decomposition of $R_{v_n}$ on the block.

\begin{lemma}\label{lem:vn}
The decomposition of the Deligne--Lusztig character $R_{v_n}$ of $G_{2n}$ associated to 
$v_n = s_1 s_2 s_1 s_2 \cdots s_{2n}$ on the principal $\Phi_{2n}$-block is given by
$$
bR_{v_n} =  \big[2n.\big] + \big[.1^{2n}\big] - \big[(2n\mn1)1.\big]-\big[.21^{2n\mn2}\big]+
(-1)^n\Big( \big[1^{n\mn2}.(n\mn1)1\big]_{B_2}+ \big[21^{n\mn2}.n\mn2\big]_{B_2} \Big).
$$
\end{lemma}

\subsection{The Deligne--Lusztig character associated to $w_n$}
Using again the notation in \cite[\S3.4.2]{GePf00}, we can write $w_n$ as $w_n = w_{(2,2n-2)}^- = 
 b_{0,2}^- b_{2,2n-2}^- = s_1 s_2 b_{2,2n-2}^-$.
Let $\blambda$ be a bipartition of $2n$. By \cite[Thm. 10.3.1]{GePf00} we have
\begin{equation}\label{eq:MNrule2}
 \chi_{\blambda}(w_n) = \sum_{\gamma} (-1)^{h(\gamma)} \chi_{\blambda \smallsetminus \gamma}(s_1s_2)
 \end{equation}
where $\gamma$ runs over all the $(2n\mn2)$-hooks of $\blambda$, and $h(\gamma)$ equals the leg length of $\gamma$ plus $1$ if $\gamma$ is in the second component.  We list in Table~\ref{tab:bipart} those bipartitions $\blambda$ which have a $(2n\mn2)$-hook $\gamma$, the value of $h(\gamma)$, the $(2n\mn2)$-core $\blambda \smallsetminus \gamma$, the value of $\chi_{\blambda \smallsetminus \gamma}(s_1s_2)$ and that of $\chi_{\blambda}(w_n)$ using \eqref{eq:MNrule2}. We do not list the ones with $(2n\mn2)$-core $1. 1$ since the corresponding character of $B_2$ vanishes on the element $s_1 s_2$. In addition, since $\chi_{\blambda^*}(w_n) = (-1)^{\ell(w_n)} \chi_{\blambda}(w_n) = \chi_{\blambda}(w_n)$ it is enough to deal with the cores $2.$ and $1^2.$.
\begin{table}
$$ \begin{array}{c|c|c|c|c}
\blambda & h(\gamma) & \blambda \smallsetminus \gamma & \chi_{\blambda \smallsetminus \gamma}(s_1s_2) & \chi_{\blambda}(w_n) \\\hline
2n. & 0 & \multirow{4}{*}{$2.$} & \multirow{4}{*}{$1$}& 1 \\
i31^{2n\mn i\mn3}. & 2n-i-2 & &&  (-1)^{i}  \\
21^{2n\mn2} .  &  2n-3 & & & -1\\ 
2 . j1^{2n\mn 2\mn j} & 2n-2 -j +1 & & & (-1)^{j-1} \\\hline 
1^{2n} . &2n-3 & \multirow{4}{*}{$1^2.$} & \multirow{4}{*}{$-1$} & 1\\
k2^21^{2n\mn k\mn4}. & 2n-i-2 & & &  (-1)^{i-1}  \\
(2n\mn 1)1. & 0& & & -1\\ 
1^2 . l1^{2n\mn 2\mn l} & 2n-2 -j +1 & & & (-1)^{j} \\  
 \end{array}$$
\caption{Bipartitions $\blambda$ with $(2n\mn2)$-core equal to $2.$ or $1^2. $}
\label{tab:bipart}
\end{table}

\smallskip

For each bipartition $\blambda$ in Table~\ref{tab:bipart} one can easily compute the family associated to each of these bipartitions and check using \S\ref{ssec:2nblock} that only those with 
$i=l = n$, $j=1$ or $n+1$ and $k=n-1$ correspond to families which have a non-trivial intersection with the principal
$\Phi_{2n}$-block $B$. The family associated to the bipartition $2n.$ has size one and contains only the trivial character. We give in Table \ref{tab:families4} the families with $4$ elements, starting with the special character in the family.

\begin{table}
$$ \begin{array}{c|c}
\blambda & \mathfrak{F} \\\hline
\vphantom{\Big)}n31^{n\mn3}. &\big[n.31^{n\mn3}\big], \big[ n31^{n\mn3}. \big], \big[2 . (n\pl1)1^{n\mn3}\big], \big[1^{n\mn3}.(n\mn1)2\big]_{B_2} \\[5pt]
21^{2n\mn2} .  & \big[2. 1^{2n\mn2} \big], [. 3 1^{2n\mn3}\big],\big[21^{2n\mn2}.\big],\big[1^{2n\mn3}.1\big]_{B_2} \\[5pt] 
2 .(n\pl1)1^{n\mn 3}  &\big[n.31^{n\mn3}\big], \big[ n31^{n\mn3}. \big], \big[2 . (n\pl1)1^{n\mn3}\big], \big[1^{n\mn3}.(n\mn1)2\big]_{B_2}\\[5pt] 
2 . 1^{2n\mn2} &   \big[2. 1^{2n\mn2} \big], [. 3 1^{2n\mn3}\big],\big[21^{2n\mn2}.\big],\big[1^{2n\mn3}.1\big]_{B_2} \\[5pt]
1^{2n} . &\big[1 . 1^{2n\mn1}\big] ,\big[1^{2n} .\big], \big[. 21^{2n\mn2}\big], [1^{2n\mn2}.\big]_{B_2} \\[5pt]
(2n\mn1)1. & \big[2n\mn1. 1\big], \big[ . 2n\big], [(2n\mn1)1. \big], \big[.2n\mn2 \big]_{B_2} \\ 
 \end{array}$$
\caption{Families $\mathfrak{F}$ with $4$ elements occurring in $bR_{w_n}$ (up to Alvis-Curtis duality)}
\label{tab:families4}
\end{table}

\smallskip

The two remaining bipartitions $(n\mn1)2^21^{n\mn3}.$ and $1^2 . n1^{n\mn2}$ correspond to unipotent characters lying in a family with $16$ elements.  Using \cite[\S4]{Lu84} one can deduce from Table~\ref{tab:families4} the almost character corresponding to each bipartition. These are listed in Table~\ref{tab:almostchar}. For the family with $16$ elements we get
$$\begin{aligned} 
bR_{\chi_{(n\mn1)2^21^{n\mn3}.}} & = -\frac{1}{4} \Big(\big[1^2.n1^{n\mn2}\big] + \big[(n\mn1)1^{n\mn1}.2\big] + \big[21^{n\mn3}.(n\mn2)1\big]_{B_2} + \big[n\mn3.1^{n\mn3}\big]_{B_6} \Big),\\ 
bR_{\chi_{1^2 . n1^{n\mn2}}} & = \frac{1}{4} \Big(\big[1^2.n1^{n\mn2}\big] + \big[(n\mn1)1^{n\mn1}.2\big] - \big[21^{n\mn3}.(n\mn2)1\big]_{B_2} - \big[n\mn3.1^{n\mn3}\big]_{B_6} \Big).\\ \end{aligned} $$ 

\begin{table}
$$ \begin{array}{c|c}
\blambda & bR_{\chi_{\blambda}} \\\hline
\vphantom{\Bigg)}n31^{n\mn3}. & \displaystyle - \frac{1}{2} \Big( \big[2 . (n\pl1)1^{n\mn3}\big] + \big[1^{n\mn3}.(n\mn1)2\big]_{B_2}\Big)  \\[5pt]
21^{2n\mn2} .  &\displaystyle \frac{1}{2} \Big( \big[21^{2n\mn2}.\big]- \big[. 3 1^{2n\mn3}\big]\Big) \\[5pt]
2 .(n\pl1)1^{n\mn 3} &\displaystyle\frac{1}{2} \Big( \big[2 . (n\pl1)1^{n\mn3}\big] - \big[1^{n\mn3}.(n\mn1)2\big]_{B_2}\Big)\\[5pt] 
2 . 1^{2n\mn2} & \displaystyle \frac{1}{2} \Big( \big[21^{2n\mn2}.\big] +  \big[. 3 1^{2n\mn3}\big]\Big)  \\[5pt] 
1^{2n} . &\displaystyle \frac{1}{2} \Big(\big[1^{2n}.\big] -  \big[. 21^{2n\mn2}\big]\Big)\\[5pt]
(2n\mn1)1. & \displaystyle \frac{1}{2} \Big( \big[(2n\mn1)1.\big]-\big[.2n\big] \Big) \\[5pt] 
 \end{array}$$
\caption{Some almost characters occurring in $bR_{w_n}$}
\label{tab:almostchar}
\end{table}

Putting this all together we obtain the decomposition of $R_{w_n}$ on the block.

\begin{lemma}\label{lem:wn}
The decomposition of the Deligne--Lusztig character $R_{w_n}$ of $G_{2n}$ associated to \\
$w_n = s_1 s_2 s_3 s_2 s_1 s_2 \cdots s_{2n}$ on the principal $\Phi_{2n}$-block is given by
$$\begin{aligned}
bR_{w_n} = & \,  \big[2n. \big] - \big[(2n\mn1)1.\big] +  \big[(2n\mn2)1^2.\big]  +\big[1^{2n}.\big] 
 + \big[.1^{2n}\big] - \big[.21^{2n\mn2}\big] +\big[.31^{2n\mn3}\big]+ \big[.2n\big] \\[4pt]
& \, + (-1)^{n-1} \Big(\big[1^{n\mn3}. (n\mn1)2\big]_{B_2} + \big[ 2^21^{n\mn3}.n\mn3\big]_{B_2}
+\big[21^{n\mn3}.(n\mn2)1\big]_{B_2} + \big[n\mn3.1^{n\mn3}\big]_{B_6} \Big).
\end{aligned}$$ 
\end{lemma}

\bibliographystyle{amsalpha}

\end{document}